\documentclass{amsart}

\usepackage{latexsym,amsfonts,amssymb,exscale,enumerate}
\usepackage{amsmath,amsthm,amscd}
\usepackage{hyperref}

\usepackage{pstricks}

\psset{linewidth=0.3pt,dimen=middle}
\psset{xunit=.70cm,yunit=0.70cm}
\psset{arrowsize=1pt 5,arrowlength=0.6,arrowinset=0.7}

\usepackage[all]{xy}
\SelectTips{cm}{}

\usepackage{graphicx}

\newcommand{\ep}{\underline{\epsilon}}
\newcommand{\onen}{{\mathbf 1}_{n}}
\newcommand{\onenn}[1]{{\mathbf 1}_{#1}}

\newcommand\rE{{\sf{E}}}

\newcommand\rF{{\sf{F}}}

\newcommand{\END}{{\rm END}}
\newcommand{\Gr}{\cat{Flag}_{N}}
\newcommand{\Grn}[1]{\cat{Flag}_{#1}}

\newcommand{\U}{\dot{{\bf U}}}
\newcommand{\Ucat}{\cal{U}}

\newcommand{\UcatD}{\dot{\cal{U}}}

\newcommand{\B}{\dot{\mathbb{B}}}

\newcommand{\UA}{{_{\cal{A}}\dot{{\bf U}}}}

\newcommand{\und}[1]{\underline{#1}}

\newcommand{\qbin}[2]{
\left[
 \begin{array}{c}
 #1 \\
 #2 \\
 \end{array}
 \right]_{q^2}
}
\newcommand{\qbins}[2]{
\left[
 \begin{array}{c}
 \scs #1 \\
 \scs #2 \\
 \end{array}
 \right]
}

\newcommand{\refequal}[1]{\xy {\ar@{=}^{#1}
(-1,0)*{};(1,0)*{}};
\endxy}

\newcommand{\cat}[1]{\ensuremath{\mbox{\bfseries {\upshape {#1}}}}}

\newcommand{\To}{\Rightarrow}

\newcommand{\Hom}{{\rm Hom}}
\newcommand{\HOM}{{\rm HOM}}
\renewcommand{\to}{\rightarrow}
\newcommand{\maps}{\colon}

\newcommand{\End}{{\rm End}}

\newcommand{\scs}{\scriptstyle}

\def\Id{\mathrm{Id}}

\def\mf{\mathfrak}
\def\shuffle{\,\raise 1pt\hbox{$\scriptscriptstyle\cup{\mskip
               -4mu}\cup$}\,}

\theoremstyle{definition}
\newtheorem{thm}{Theorem}[section]
\newtheorem{cor}[thm]{Corollary}

\newtheorem{rem}[thm]{Remark}
\newtheorem{prop}[thm]{Proposition}
\newtheorem{defn}[thm]{Definition}

\numberwithin{equation}{section}

\def\emph#1{{\sl #1\/}}

\let\hat=\widehat

\let\phi=\varphi
\let\theta=\vartheta
\let\epsilon=\varepsilon

\usepackage{bbm}
\def\C{{\mathbbm C}}
\def\N{{\mathbbm N}}

\def\Z{{\mathbbm Z}}
\def\Q{{\mathbbm Q}}

\def\cal#1{\mathcal{#1}}%
\def\1{\mathbbm{1}}%
\def\nn{\notag}

\def\la{\langle}
\def\ra{\rangle}

\newcommand{\Uup}{
    \xy {\ar (0,-3)*{};(0,3)*{} };(1.5,0)*{};(-1.5,0)*{};\endxy}

\newcommand{\Udown}{
    \xy {\ar (0,3)*{};(0,-3)*{} };(1.5,0)*{};(-1.5,0)*{};\endxy}

\newcommand{\Uupdot}{
   \xy {\ar (0,-3)*{};(0,3)*{} };(0,0)*{\bullet};(1.5,0)*{};(-1.5,0)*{};\endxy}

\newcommand{\Udowndot}{
   \xy {\ar (0,3)*{};(0,-3)*{} };(0,0)*{\bullet};(1.5,0)*{};(-1.5,0)*{};\endxy}

\newcommand{\Ucupr}{\;\;
    \vcenter{\xy (-2,3)*{}; (2,3)*{} **\crv{(-2,-1) & (2,-1)}?(1)*\dir{>};
            (2,-3)*{};(-2,3)*{}; \endxy} \;\; }

\newcommand{\Ucapl}{\;\;
    \vcenter{\xy (2,-3)*{}; (-2,-3)*{} **\crv{(2,1) & (-2,1)}?(1)*\dir{>};
            (2,-3)*{};(-2,3)*{}; \endxy} \;\; }

\newcommand{\ccbub}[1]{
\xybox{%
 (-6,0)*{};
  (6,0)*{};
  (-4,0)*{}="t1";
  (4,0)*{}="t2";
  "t2";"t1" **\crv{(4,6) & (-4,6)};
   ?(1)*\dir{>};
  "t2";"t1" **\crv{(4,-6) & (-4,-6)};
   ?(.3)*\dir{}+(0,0)*{\bullet}+(0,-3)*{\scs {#1}};
}}

\newcommand{\cbub}[1]{
\xybox{%
 (-6,0)*{};
  (6,0)*{};
  (-4,0)*{}="t1";
  (4,0)*{}="t2";
  "t2";"t1" **\crv{(4,6) & (-4,6)};
    ?(.95)*\dir{<};
  "t2";"t1" **\crv{(4,-6) & (-4,-6)};
   ?(.3)*\dir{}+(0,0)*{\bullet}+(0,-3)*{\scs {#1}};
}}

\newcommand{\bbpef}{\xybox{%
  (-6,0)*{};
  (6,0)*{};
  (-4,0)*{}="t1";
  (4,0)*{}="t2";
  "t1";"t2" **\crv{(-4,-6) & (4,-6)}; ?(.15)*\dir{>} ?(.9)*\dir{>};
}}
\newcommand{\bbpfe}{\xybox{%
  (-6,0)*{};
  (6,0)*{};
  (-4,0)*{}="t1";
  (4,0)*{}="t2";
  "t2";"t1" **\crv{(4,-6) & (-4,-6)}; ?(.15)*\dir{>} ?(.9)*\dir{>};
}}

\newcommand{\bbcfe}[1]{\xybox{%
  (-6,0)*{};
  (6,0)*{};
  (-4,0)*{}="t1";
  (4,0)*{}="t2";
  "t1";"t2" **\crv{(-4,6) & (4,6)}; ?(.15)*\dir{>} ?(.9)*\dir{>}
  ?(.5)*\dir{}+(0,2)*{\scriptstyle{#1}};
}}
\newcommand{\bbcef}[1]{\xybox{%
  (-6,0)*{};
  (6,0)*{};
  (-4,0)*{}="t1";
  (4,0)*{}="t2";
  "t2";"t1" **\crv{(4,6) & (-4,6)}; ?(.15)*\dir{>}
  ?(.9)*\dir{>} ?(.5)*\dir{}+(0,2)*{\scriptstyle{#1}};
}}

\newcommand{\lowrru}[1]{\xybox{%
  (-8,0)*{};
  (8,0)*{};
  (-6,-18)*{};(6,-9)*{} **\crv{(-6,-13) & (6,-15)} ?(1)*\dir{>};
  (6,-9)*{};(6,0)*{}  **\dir{-} ?(.3)*\dir{ }+(2,0)*{\scs {\bf j}};
}}

\newcommand{\lowllu}[1]{\xybox{%
  (-8,0)*{};
  (8,0)*{};
  (6,-18)*{};(-6,-9)*{} **\crv{(6,-13) & (-6,-15)} ?(1)*\dir{>};
  (-6,-9)*{};(-6,0)*{}  **\dir{-} ?(.3)*\dir{ }+(-2,0)*{\scs {\bf j}};
}}

\newcommand{\bbdl}[1]{\xybox{%
  (2,0);(0,-8) **\crv{(2,-2)&(0,-6)}; ?(.5)*\dir{>}
}}
\newcommand{\bbdlu}[1]{\xybox{%
  (2,0);(0,-8) **\crv{(2,-2)&(0,-6)}; ?(.5)*\dir{<}
}}
\newcommand{\bbdr}[1]{\xybox{%
  (-2,0);(0,-8) **\crv{(-2,-2)&(0,-6)}; ?(.5)*\dir{>}
}}
\newcommand{\bbdru}[1]{\xybox{%
  (-2,0);(0,-8) **\crv{(-2,-2)&(0,-6)}; ?(.5)*\dir{<}
}}

\newgray{whitegray}{.9}

\allowdisplaybreaks

\begin{document}
%

\title{Categorified quantum $\mathfrak{sl}_2$ is an inverse limit of flag 2-categories}

\author{Anna Beliakova}
\address{Universit\"at Z\"urich, Winterthurerstr. 190
CH-8057 Z\"urich, Switzerland}
\email{anna@math.uzh.ch}

\author{Aaron D. Lauda}
\address{Department of Mathematics, University of Southern California, Los Angeles, CA 90089, USA}
\email{lauda@usc.edu}

\begin{abstract}
We prove that categorified quantum $\mathfrak{sl}_2$ is an inverse limit of Flag 2-categories defined using cohomology rings of iterated flag varieties.   This inverse limit is an instance of a 2-limit in a bicategory giving rise to a universal property that characterizes the categorification of quantum $\mathfrak{sl}_2$ uniquely up to equivalence.  As an application we characterize all bimodule homomorphisms in the Flag 2-category and prove that the categorified quantum Casimir of $\mathfrak{sl}_2$ acts appropriately on these 2-representations.
\end{abstract}

\maketitle


\newcommand{\uk}{\underline{k}}

In \cite{Lau1} the second author  categorified   quantum ${\mathfrak {sl}}_2$.
For this purpose, a 2-category $\dot\Ucat$ was introduced as an idempotent completion
of an additive 2-category $\Ucat$. The objects $n\in \Z$ of $\Ucat$ are parameterized
by the integral weight lattice of ${\mathfrak {sl}}_2$. The 1-morphisms
are generated by $\cal{E}\onen$ and $\cal{F}\onen$, which can be thought of as
categorifications of the generators $E$ and $F$ of quantum ${\mathfrak {sl}}_2$.
The 2-morphisms are $\Bbbk$-linear combinations of certain planar
diagrams modulo local relations. The Grothendieck group of  $\dot\Ucat$ coincides with
the integral idempotented version $\U$ of the quantum enveloping algebra of
${\mathfrak {sl}}_2 $ with generic quantum parameter $q$ \cite{BLM}.
In joint work with Khovanov, this approach was extended to all quantum groups~\cite{KL,KL2,KL3}. Related categorifications were developed by Chuang and Rouquier~\cite{CR,Rou2}. The precise relationship between these approaches is explained in \cite{CLau}.

The main categorification result in \cite{Lau1}  relies on the existence and faithfullness (for $N$ large) of  the  2-functor $\Gamma_N$ from $\UcatD$ to a 2-category $\Gr$. The  objects of $\Gr$ are cohomology rings of Grassmannians in $\C^N$, the morphisms are generated by bimodules arising from cohomology rings of iterated flag varieties, and the 2-morphisms are bimodule homomorphisms.  Variants of this 2-category have been studied previously in connection with higher representation theory~\cite{CR,FKS}.

In this article we work with a refined version of the 2-category $\Gr$.  The original 2-category is described as an idempotent completion of a 2-category built from cohomology rings of one-step flag varieties.  Here we add 1-morphisms in the form of bimodules given by cohomology rings of additional partial flag varieties. With these additional 1-morphisms we give a direct proof of idempotent completeness of $\Gr$ (Theorem~\ref{thm_Karoubi}).  This careful analysis allows us to prove that the 2-functors $\Gamma_N$ are also full in an appropriate sense (Theorem~\ref{thm_full}).  This implies that all bimodule homomorphisms between cohomology rings of iterated flag varieties and their tensor products are described by the 2-category $\UcatD$.  Any such bimodule homomorphism is a  composition of images of
 the generating 2-morphisms modulo the relations  in the 2-category $\UcatD$. Hence, the study of bimodule homomorphisms in $\Gr$
is encoded into the diagrammatic calculus of $\UcatD$ that has been developed in \cite{KLMS}. The  main technique used in the proof is manipulation of the Hom spaces using the existence of biadjoints as in \cite{CLau}.  This allows us to compute explicitly the images of the Hom-spaces in $\UcatD$ under $\Gamma_N$.

The 2-morphisms in $\UcatD$ carry a natural grading making the space of 2-morphisms between two 1-morphisms into a graded $\Bbbk$-vector space.  Let us denote by $\Hom^a(x,y)$ the space of all degree $a$ 2-morphisms in $\UcatD$ from $x$ to $y$. The results described above imply that the 2-functor $\Gamma_N:\UcatD \to \Gr$ is locally full and eventually faithful
in the sense that given any two 1-morphisms $x$ and $y$ of $\UcatD$ and  $d\in \N$, there exists
a positive integer $M$, such that the map
 \[
 \Gamma_N \maps \Hom^d_{\UcatD} (x,y) \to \Hom_{\Gr \nn}^d (\Gamma_N(x), \Gamma_N(y))
  \]
  is an isomorphisms for any $N>M$.

This observation suggests it should be possible to realize the  2-category $\UcatD$ as an inverse limit of the 2-categories $\Gr$.  Our main result makes this precise by proving that the 2-category $\UcatD$ is a 2-limit in an appropriate bicategory (Theorem~\ref{thm_equiv}).  This characterization of $\UcatD$ gives rise to a universal property describing $\UcatD$ uniquely up to equivalence.  It should be straightforward to generalize this result to the 2-category $\Ucat(\mf{sl}_n)$ and the related 2-functors $\Gamma_N$ defined in \cite{KL3}. Part of our interest in such a description stems from conversations with Ben Cooper about his forthcoming preprint.  This result can be viewed as a categorical analog of the classical result of Beilinson, Lusztig, and MacPherson~\cite{BLM} realizing the idempotent form of the quantum group as an inverse limit.

The integral idempotent version of the Casimir element of quantum $\mathfrak{sl}_2$ has the form
\begin{eqnarray}
 \dot{C} &=& \prod_{n\in\Z}C1_n, \nn \\
  C1_n =1_nC &:=& (-q^2+2-q^{-2})FE1_n-(q^{n+1}+q^{-1-n})1_n. \nn
\end{eqnarray}
Previous work \cite{BKL} of the authors together with Khovanov defines a complex categorifying
this form of the Casimir element.  This complex exists in the 2-category $\rm Kom(\Ucat)$ with the same objects as $\Ucat$, but whose morphisms are chain complexes of morphisms from $\Ucat$; the 2-morphisms in $\rm Kom (\Ucat)$ are chain maps built from the diagrammatically defined 2-morphisms in $\Ucat$. Just as the center of quantum ${\mathfrak{ sl}}_2$
as an algebra is generated by the Casimir operator, in \cite{BKL} it is shown that the Casimir complex lies in the (Drinfeld) center of the 2-category $\rm Kom(\Ucat)$.  In particular, this complex commutes with other complexes up to chain homotopy equivalence and possesses desirable naturality properties with respect to 2-morphisms in $\Ucat$.  The Casimir complex reduces to the integral idempotent form of the Casimir operator when passing to the Grothendieck ring so that the Casimir complex can be viewed as a categorification of the Casimir element.  We expect the Casimir complex to play an important role in categorified representation theory.

In this article, we make an additional connection between the categorified Casimir element and its decategorification.  The Casimir element acts by a multiple of the identity on any highest-weight representation.   Here we compute the action of the Casimir complex
on $\Gr\nn$ and show that
${\cal C}1_n$ acts on a categorified $(N+1)$ dimensional representation
as it should: namely by producing two copies of the identity on any weight space
with  degree shifts  $(N+1)$ and $(-N-1)$, respectively, see Theorem \ref{thm_cas}. However, in
$\rm Kom(\Ucat)$, the complexes ${\cal C}1_n$ and $1_n\{N+1\}[-1]\oplus 1_n\{-N-1\}[1]$ are not  homotopy equivalent.

We expect that the computation of the action of the Casimir complex on $\Gr$, together with the description of $\UcatD$ as an inverse limit,  will help to determine whether the  center of $\rm Kom (\Ucat)$ has a structure of a braided monoidal category with non--trivial braiding.
Our motivation for studying the center of the 2-category ${\rm Kom}(\UcatD)$ stems from the relationship between the center of quantum $\mathfrak{sl}_2$ and the universal ${\mathfrak{ sl}}_2$ knot invariant~\cite{Hab,Law2,Law}.
\medskip

\noindent {\bf Acknowledgments:}
The authors are grateful to Eugenia Cheng for helpful discussions on limits in higher categories. Both authors would also like to thank Ben Cooper for interesting conversations related to this paper.
The first author was partially supported by the Swiss National
Science Foundation.
The second author was partially supported by the NSF grant DMS-0855713 and the Alfred P. Sloan foundation.


%
\section{The 2-category $\UcatD$}
%

By a graded category we will mean a category equipped with an auto-equivalence $\la 1 \ra$. We denote by $\la l \ra$ the auto-equivalence obtained by applying $\la 1 \ra$ $l$ times. If $x,y$ are two objects then $\Hom^s(x,y)$ will be short hand for the space  $\Hom(x,y \la s \ra)$ of degree $s$ morphisms.  In particular, we write
\[
 \Hom(x \la s \ra, y \la t \ra ) =  \Hom(x, y \la t-s \ra ) =\Hom(x\la s-t\ra, y )=\Hom^{t-s}(x,y).
\]
Define graded $\Bbbk$-vector spaces $\HOM(x,y) := \bigoplus_{s \in \Z} \Hom^s(x,y)$ and $\END(x)=\HOM(x,x)$.

A graded additive $\Bbbk$-linear 2-category is a category enriched over graded additive $\Bbbk$-linear categories, that is a 2-category $\cal{K}$ such that the Hom categories $\Hom_{\cal{K}}(x,y)$ between objects $x$ and $y$ are graded additive $\Bbbk$-linear categories and the composition maps $\Hom_{\cal{K}}(x,y) \times \Hom_{\cal{K}}(y,z) \to \Hom_{\cal{K}}(x,z)$ form a graded additive $\Bbbk$-linear functor. Given a 1-morphism $x$ in an additive 2-category $\cal{K}$ and a Laurent polynomial $f=\sum f_a q^a\in \N[q,q^{-1}]$ we write $\oplus_f x$ or $x^{\oplus f}$ for the direct sum over $a\in \N$, of $f_a$ copies of $x\la a\ra$.
Here we will use the terminology 2-category, 2-functor, and 2-natural transformations in place of bicategory, pseudo functor, and pseudo natural transformation, see \cite[Section 7.7]{Bor}.  A 2-category is called strict when 1-morphisms compose strictly associatively, and a 2-functor is called strict when it preserves composition of 1-morphisms on the nose. Some of these definitions are recalled in section~\ref{sec_Lack}.

We briefly recall the (strict) 2-category $\Ucat$ from \cite{Lau1}.  The definition below follows most closely \cite{CLau}.  For more on the categorification of quantum $\mathfrak{sl}_2$ see \cite{Lau2, Lau3,KL3,Web}.

\begin{defn} \label{defU_cat}
The strict 2-category $\Ucat=\Ucat(\mf{sl}_2)$ is the graded additive $\Bbbk$-linear 2-category consisting of
\begin{itemize}
\item objects $n$ for $n \in \Z$,
\item 1-morphisms are formal direct sums of compositions of
$$\onen, \quad \onenn{n+2} \cal{E} = \onenn{n+2} \cal{E}\onen = \cal{E} \onen \quad \text{ and }\quad \onenn{n-2} \cal{F} = \onenn{n-2} \cal{F}\onen = \cal{F}\onen$$
for $n \in \Z$, together with their grading shifts $x\la t\ra$ for all 1-morphisms $x$ and $t \in \Z$, and
\item 2-morphisms are $\Bbbk$-vector spaces spanned by (vertical and horizontal) compositions of tangle-like diagrams illustrated below.
\begin{align}
  \xy 0;/r.17pc/:
 (0,7);(0,-7); **\dir{-} ?(.75)*\dir{>};
 (0,0)*{\bullet};
 (7,3)*{ \scs n};
 (-9,3)*{\scs n+2};
 (-10,0)*{};(10,0)*{};
 \endxy &\maps \cal{E}\onen \to \cal{E}\onen\la 2 \ra  & \quad
 &
    \xy 0;/r.17pc/:
 (0,7);(0,-7); **\dir{-} ?(.75)*\dir{<};
 (0,0)*{\bullet};
 (7,3)*{ \scs n};
 (-9,3)*{\scs  n-2};
 (-10,0)*{};(10,0)*{};
 \endxy\maps \cal{F}\onen \to \cal{F}\onen\la 2 \ra  \nn \\
   & & & \nn \\
   \xy 0;/r.17pc/:
  (0,0)*{\xybox{
    (-4,-4)*{};(4,4)*{} **\crv{(-4,-1) & (4,1)}?(1)*\dir{>} ;
    (4,-4)*{};(-4,4)*{} **\crv{(4,-1) & (-4,1)}?(1)*\dir{>};
     (9,1)*{\scs  n};
     (-10,0)*{};(10,0)*{};
     }};
  \endxy \;\;&\maps \cal{E}\cal{E}\onen  \to \cal{E}\cal{E}\onen\la -2 \ra  &
  &
   \xy 0;/r.17pc/:
  (0,0)*{\xybox{
    (-4,4)*{};(4,-4)*{} **\crv{(-4,1) & (4,-1)}?(1)*\dir{>} ;
    (4,4)*{};(-4,-4)*{} **\crv{(4,1) & (-4,-1)}?(1)*\dir{>};
     (9,1)*{\scs  n};
     (-10,0)*{};(10,0)*{};
     }};
  \endxy\;\; \maps \cal{F}\cal{F}\onen  \to \cal{F}\cal{F}\onen\la -2 \ra  \nn \\
  & & & \nn \\
     \xy 0;/r.17pc/:
    (0,-3)*{\bbpef{}};
    (8,-5)*{\scs  n};
    (-10,0)*{};(10,0)*{};
    \endxy &\maps \onen  \to \cal{F}\cal{E}\onen\la 1+n \ra   &
    &
   \xy 0;/r.17pc/:
    (0,-3)*{\bbpfe{}};
    (8,-5)*{\scs n};
    (-10,0)*{};(10,0)*{};
    \endxy \maps \onen  \to\cal{E}\cal{F}\onen\la 1-n\ra  \nn \\
      & & & \nn \\
  \xy 0;/r.17pc/:
    (0,0)*{\bbcef{}};
    (8,4)*{\scs  n};
    (-10,0)*{};(10,0)*{};
    \endxy & \maps \cal{F}\cal{E}\onen \to\onen\la 1+n \ra  &
    &
 \xy 0;/r.17pc/:
    (0,0)*{\bbcfe{}};
    (8,4)*{\scs  n};
    (-10,0)*{};(10,0)*{};
    \endxy \maps\cal{E}\cal{F}\onen  \to\onen\la 1-n \ra \nn
\end{align}
The degree of the 2-morphisms can be read from the shift on the right-hand side.
\end{itemize}
\end{defn}
Diagrams are read from right to left and bottom to top.  The identity 2-morphism of the 1-morphism $\cal{E} \onen$ is represented by an upward oriented line (likewise, the identity 2-morphism of $\cal{F} \onen$ is represented by a downward oriented line).

The 2-morphisms satisfy the following relations.
\begin{enumerate}
\item \label{item_cycbiadjoint} The 1-morphisms $\cal{E} \onen$ and $\cal{F} \onen$ are biadjoint (up to a specified degree shift). Moreover, the 2-morphisms are cyclic with respect to this biadjoint structure.

\item The $\cal{E}$'s carry an action of the nilHecke algebra. Using the adjoint structure this induces an action of the nilHecke algebra on the $\cal{F}$'s.

\item \label{item_positivity} Dotted bubbles of negative degree are zero, so that
for all $m \in \Z_{\geq 0}$ one has
\begin{equation} \label{eq_positivity_bubbles}
\xy 0;/r.18pc/:
 (-12,0)*{\cbub{m}{}};
 (-8,8)*{n};
 \endxy
  = 0
 \qquad  \text{if $m< n-1$,} \qquad \xy 0;/r.18pc/: (-12,0)*{\ccbub{m}{}};
 (-8,8)*{n};
 \endxy = 0\quad
  \text{if $m< -n -1$.}
\end{equation}
Dotted bubbles of degree zero are equal to the identity 2-morphisms:
\[
\xy 0;/r.18pc/:
 (0,0)*{\cbub{n-1}};
  (4,8)*{n};
 \endxy
  =  \Id_{\onen} \quad \text{for $n \geq 1$,}
  \qquad \quad
  \xy 0;/r.18pc/:
 (0,0)*{\ccbub{-n-1}};
  (4,8)*{n};
 \endxy  =  \Id_{\onen} \quad \text{for $n \leq -1$.}\]

\item \label{item_highersl2}
There are additional relations requiring that the $\mf{sl}_2$ relations lift to isomorphisms in $\Ucat$.  In particular, if $n \geq 0$ the 2-morphism
\begin{equation}\label{eq_phi-minus}
\xy 0;/r.15pc/:
    (-4,-4)*{};(4,4)*{} **\crv{(-4,-1) & (4,1)}?(0)*\dir{<} ;
    (4,-4)*{};(-4,4)*{} **\crv{(4,-1) & (-4,1)}?(1)*\dir{>};
  \endxy \;\; \bigoplus_{k=0}^{n-1}
\vcenter{\xy 0;/r.15pc/:
 (-4,2)*{}="t1"; (4,2)*{}="t2";
 "t2";"t1" **\crv{(4,-5) & (-4,-5)}; ?(1)*\dir{>}
 ?(.8)*\dir{}+(0,-.1)*{\bullet}+(-3,-1)*{\scs k};;
 \endxy}: \cal{F} \cal{E} \onen \bigoplus_{k=0}^{n-1} \onen \la n-1-2k \ra \rightarrow \cal{E} \cal{F} \onen
\end{equation}
is invertible in $\Ucat$, and if $n \le 0$ then the 2-morphism
\begin{equation}\label{eq_phi-plus}
\xy 0;/r.15pc/:
    (-4,-4)*{};(4,4)*{} **\crv{(-4,-1) & (4,1)}?(1)*\dir{>} ;
    (4,-4)*{};(-4,4)*{} **\crv{(4,-1) & (-4,1)}?(0)*\dir{<};
  \endxy \;\; \bigoplus_{k=0}^{-n-1}
\vcenter{\xy 0;/r.15pc/:
 (4,2)*{}="t1"; (-4,2)*{}="t2";
 "t2";"t1" **\crv{(-4,-5) & (4,-5)}; ?(1)*\dir{>}
 ?(.8)*\dir{}+(0,-.1)*{\bullet}+(3,-1)*{\scs k};;
 \endxy}:  \cal{E} \cal{F} \onen \bigoplus_{k=0}^{-n-1} \onen \la -n-1-2k \ra \rightarrow \cal{F} \cal{E}  \onen
\end{equation}
is invertible in $\Ucat$.
\end{enumerate}

In \cite{CLau} it is shown that any choice of inverses for the 2-morphisms giving the $\mf{sl}_2$-relations that are compatible with the structure above produces a 2-category isomorphic to the 2-category $\Ucat$ from \cite{Lau1}. Therefore, we omit the specific relations here. For more details see \cite[Section 3]{Lau3} and \cite{CLau}.

The 2-category $\UcatD$ is the Karoubi completion (or idempotent completion) of the 2-category $\Ucat$.  In $\UcatD$ there are 1-morphisms $\cal{E}^{(a)}\onen$ and $\cal{F}^{(b)}\onen$ satisfying relations $\cal{E}^a\onen = \oplus_{[a]!} \cal{E}^{(a)}\onen$ and $\cal{F}^{(b)}\onen= \oplus_{[b]!}\cal{F}^{b}\onen$, where $[n]:=(q^n-q^{-n})/(q-q^{-1})$ and $[n]:=[n][n-1]\dots[1]$.

It is convenient to describe 1-morphisms in $\Ucat$ using ordered sequences
$\ep = (\epsilon_1,\epsilon_2, \dots, \epsilon_m)$ of elements $\epsilon_1, \dots, \epsilon_m \in \{ +,-\}$ called signed sequences.  Write $\cal{E}_{\ep} \onen := \cal{E}_{\epsilon_1} \dots \cal{E}_{\epsilon_m}$ with $\cal{E}_{+} := \cal{E}$ and $\cal{E}_{-} := \cal{F}$.  Likewise, 1-morphisms in $\UcatD$ can be described by
divided powers signed sequences
\begin{equation}
  (\ep) = (\epsilon_1^{(a_1)},\epsilon_2^{(a_2)},\dots,\epsilon_m^{(a_m)})
\end{equation}
where $\epsilon$'s are as before and $a_1, \dots, a_m \in \{
1,2,\dots\}$, so that $\cal{E}_{(\ep)}\onen = \cal{E}_{\epsilon_1}^{(a_1)}
\dots \cal{E}_{\epsilon_m}^{(a_m)}\onen$.
Set $|\ep|= \sum_{k=1}^m \epsilon_k a_k  \in \Z$.

Another convenient notation indexes the 1-morphisms in $\UcatD$ that lift elements of Lusztig's canonical basis $\B$ of $\U$
\begin{enumerate}[(i)]
     \item $E^{(a)}F^{(b)}1_{n} \quad $ for $a$,$b\in \Z_{\geq 0}$,
     $n\in\Z$, $n\leq b-a$,
     \item $F^{(b)}E^{(a)}1_{n} \quad$ for $a$,$b\in\Z_{\geq 0}$, $n\in\Z$,
     $n\geq
     b-a$,
\end{enumerate}
where $E^{(a)}F^{(b)}1_{b-a}=F^{(b)}E^{(a)}1_{b-a}$.   To each $x \in \B$ we associate a 1-morphism in $\UcatD$:
\begin{equation} \label{eq_basis}
  x \mapsto \cal{E}(x) := \left\{
\begin{array}{cl}
  \cal{E}^{(a)}\cal{F}^{(b)}\onen & \text{if $x=E^{(a)}F^{(b)}1_n$,} \nn \\
  \cal{F}^{(b)}\cal{E}^{(a)}\onen & \text{if $x=F^{(b)}E^{(a)}1_n$.} \nn
\end{array}
  \right.
\end{equation}

Let $\UA$ denote the $\Z[q,q^{-1}]$ form of the idempotented algebra $\U$.  The defining relations for the algebra $\UA$ all lift to explicit isomorphisms in $\UcatD$ (see \cite[Theorem 5.1, Theorem 5.9]{KLMS}) giving rise to an isomorphism of $\Z[q,q^{-1}]$-modules
\begin{eqnarray} \label{def_gamma}
\gamma\maps  \UA & \longrightarrow& K_0(\UcatD) \\
   x & \mapsto & [\cal{E}(x)], \nn
\end{eqnarray}
where $\UcatD$ is defined over a commutative ring $\Bbbk$.  In this article $\Bbbk$ can be any field in which $2$ is invertible. We set $\Bbbk=\Q$ for convenience.

%
\section{The Flag 2-category}
%

%
\subsection{Cohomology of partial flag varieties} \label{eq_cohomology}
%
Fix a positive integer $N$.  For $0 \leq k \leq
N$ let $G_k^N$ denote the variety of complex k-planes in $\C^N$. Given a sequence of integers $\uk = (k_1, k_2, \dots, k_m)$ with $0 \leq k_i \leq N$ there is a partial flag variety $G_{\uk}^N =G_{k_1,\ldots,k_m}^N$ consisting of sequences $(W_1,\ldots,W_m)$ of linear subspaces of $\C^N$ such that the dimension of $W_i$ is $k_i$ and $W_i \subset W_{i+1}$
if $k_i \leq k_{i+1}$ and $W_i \supset W_{i+1}$ if $k_i \geq k_{i+1}$.  In particular, $G_{\uk}^N$ is a Grassmannian when $m=1$.

Let $H_{\uk;N}$ denote the cohomology ring of $G_{\uk}^N$. The forgetful maps
\[
 \xymatrix{G_{k_1}^N & G_{\uk}^N \ar[l]_{p_1} \ar[r]^{p_2} & G_{k_m}^N}
\]
induce maps of cohomology rings
\begin{equation} \label{eq_cohom_inclusions}
 \xymatrix{ H_{k_1;N} \ar[r]^-{p_1^*} & H_{\uk;N}   & H_{k_m;N} \ar[l]_-{p_2^*}}
\end{equation}
which make the cohomology ring $H_{\uk;N}$ into a
graded $(H_{k_1;N},H_{k_m;N})$-bimodule.

When the sequence $\uk$ is an increasing sequence, so that $0 \leq k_1 < k_2 < \dots < k_m \leq N$, then there is an isomorphism of graded bimodules
\begin{align} \label{eq_Hk-tensor}
  H_{\uk;N} \cong
  H_{k_1,k_2;N} \otimes_{H_{k_2;N}} H_{k_2,k_3;N} \otimes_{H_{k_3;N}}
   \dots \otimes_{H_{k_{m-1};N} H_{k_{m-1},k_m;N}}.
\end{align}
There is a similar isomorphism when $\uk$ is a decreasing sequence.

It is possible to give an explicit presentation of the cohomology ring $H_{\uk;N}$ as a quotient of a graded polynomial ring by a homogeneous ideal.   Useful references for this material are \cite[Chapter 3]{Hiller} and \cite[Chapter 10]{Fulton}, see also \cite[Section 6]{Lau1}.  We recall some of these presentations to fix our notation.  In what follows we introduce a new parameter $n=2k-N$ called the weight.
\begin{itemize}
  \item The cohomology of the Grassmannian is given by
  \begin{equation} \nn
H_{k;N}:=
\Q[x_{1,n}, x_{2,n},\dots, x_{k,n};
y_{1,{n}}, \dots, y_{N-k, n}]/ I_{k;N}
\end{equation}
with $I_{k;N}$ the homogeneous ideal generated by equating the homogeneous terms in $t$ in the equation
\begin{equation} \nn
  (1+x_{1,n}t+\dots+x_{k,n}t^k)(1+y_{1,n}t+
\dots+y_{N-k,n} t^{N-k})=1.
\end{equation}

\item The cohomology of the $a$th iterated 1-step flag variety $H_{\uk;N}$ with $\uk=(k,k+1,k+2,\dots, k+a)$ is given by
\begin{equation} \nn
H_{\uk;N}:=
\Q[x_{1,n}, x_{2,n},\dots, x_{k,n};\xi_1,\dots, \xi_a;
y_{1,{n+2a}}, \dots, y_{N-k-a, n+2a}]/ I_{\uk;N}
\end{equation}
with $I_{\uk;N}$ the homogeneous ideal generated by equating the homogeneous terms in $t$ in the equation
\begin{equation} \nn
  (1+x_{1,n}t+\dots+x_{k,n}t^k)(1+\xi_1 t)(1+\xi_2 t)\dots (1+\xi_a t)(1+y_{1,n+2a}t+
\dots+y_{N-k-a,n+2a} t^{N-k-a})=1.
\end{equation}

\item The cohomology of $a$-step flag variety corresponding to the sequence $\uk=(k,k+a)$ is given by
\begin{equation} \nn
H_{k,k+a;N}:=
\Q[x_{1,n}, x_{2,n},\dots, x_{k,n};
\varepsilon_1,\dots, \varepsilon_a;
y_{1,{n+2a}}, \dots, y_{N-k-a, n+2a}]/ I_{k,k+a;N}
\end{equation}
with $I_{k,k+a;N}$ the homogeneous ideal generated by equating the homogeneous terms in $t$ in the equation
\begin{equation} \nn
  (1+x_{1,n}t+\dots+x_{k,n}t^k)(1+\varepsilon_1 t+\dots +\varepsilon_a t^a)(1+y_{1,n+2a}t+
\dots+y_{N-k-a,n+2a} t^{N-k-a})=1.
\end{equation}
\end{itemize}

Our choice of subscripts on the variables arise from the inclusions \eqref{eq_cohom_inclusions} of the cohomology rings of Grassmannians into the cohomology of the partial flag varieties.  With this notation the bimodule structure is apparent.  These rings are graded with $\deg(x_{j,n}) =2 j=\deg(y_{j,n+2a})$, $\deg(\xi_j)=2$, and $\deg(\varepsilon_j)=2j$.

For a fixed value of $N$ and $n=2k-N$ we set $x_{j,n}=0$ for $j<0$ or $j>k$ and set $y_{\ell,n}=0$ for $\ell<0$ and $\ell>N-k$.  The variable $x_{0,n}$ and $y_{0,n}$ are always set equal to 1.

%
%

Note that the graded dimension of the cohomology ring of the iterated flag variety  is given by the following formula.
\begin{equation}
{\rm rk}_q H_{k,k+1,\dots, k+a}=\qbin{N}{k} \qbin{N-k}{a}=\qbin{N}{k+a} \qbin{k+a}{a} \label{grdim1}
\end{equation}
where $\qbin{N}{k}:=\frac{\{N\}^!_{q^2}}{\{k\}^!_{q^2}\{N-k\}^!_{q^2}}$
and $\{a\}_{q^2}:=1+q^2+\dots+q^{2(a-1)}$.

%
\subsection{The definition of the 2-category $\Gr$}
%

We now define a sub-2-category $\Gr$ of the 2-category $\cat{Bim}$ consisting of graded rings, graded bimodules, and degree preserving bimodule homomorphisms.  The 2-category $\Gr$ is a 2-full sub-2-category meaning that we define a subset of objects and morphisms, but allow all 2-morphisms between these chosen 1-morphisms.

\begin{defn}
The 2-category $\Gr$ is the graded additive 2-category with
\begin{itemize}
  \item objects: rings $H_{k,N}$, where $k \in \Z$ and $H_{k,N}$ is the zero ring for $k$ not between $0$ and $N$,
  \item 1-morphisms: arbitrary bimodule tensor products of the bimodules $H_{\uk;N}$ for all strictly increasing or strictly decreasing sequences $\uk$, and
  \item 2-morphisms: all degree-preserving bimodule homomorphisms.
\end{itemize}
We take the graded additive closure of all Hom categories so that we also allow formal direct sums of these tensor products and add additional grading shifted objects $x \la s\ra$ and grading shifted 2-morphisms $f\la s \ra \maps x\la s\ra \to y \la s \ra$ for each 1-morphism $x$, 2-morphism $f \maps x \to y$, and each $s \in \Z$.   A 2-morphism of $f \maps x \to y$ of degree $s$ is equivalent to a degree-preserving bimodule homomorphism $x \to y \la s \ra$.
\end{defn}

Note that by \eqref{eq_Hk-tensor} every 1-morphism in $\Gr$ is isomorphic to a direct sum of tensor products of bimodules $H_{\uk;N}$ for the sequences $\uk=(k,k+a)$ and $\uk=(k+a,k)$ for $0 \leq k \leq N$ and $0 \leq a\leq N-k$.

\begin{rem}
In \cite{Lau1} the 2-category $\Gr$ was defined as the Karoubi completion of the 2-category defined above for the sequences $\uk=(k,k+1)$ and $\uk=(k+1,k)$ for $0 \leq k <N$.  In this article we show that it is not necessary to take the Karoubi envelope if we consider all increasing and decreasing sequences.
\end{rem}

%
\subsection{The 2-functor $\Gamma_N$}
%

Let us recall the definition of the 2-functor $\Gamma_N$ from \cite[Section 7]{Lau1}. On objects the 2-functor $\Gamma_{N}$ sends
$n$ to the ring $H_{k;N}$ whenever $n$ and $k$ are compatible:
\begin{eqnarray}
 \Gamma_{N} \maps \Ucat & \to & \Gr \nn \\
 n & \mapsto &
  \left\{\begin{array}{ccl}
    H_{k;N} & & \text{with $n=2k-N$ and $0\leq k \leq N$,} \\
    0  & & \text{otherwise.}
  \end{array} \right.
\end{eqnarray}
Morphisms of $\Ucat$ get mapped by $\Gamma_N$ to graded bimodules
\begin{eqnarray} \label{eq_GammaN_def}
 \Gamma_{N} \maps \Ucat & \to & \Gr  \\
  \onen\la s\ra & \mapsto &
  \left\{\begin{array}{ccl}
    H_{k;N}\la s\ra & & \text{with $n=2k-N$ and $0\leq k \leq N$,} \\
    0  & & \text{otherwise.}
  \end{array} \right.  \nn \\
  \cal{E}\onen\la s\ra & \mapsto &
  \left\{\begin{array}{ccl}
    H_{k+1,k;N}\la s+1-N+k\ra & & \text{with $n=2k-N$ and $0\leq k < N$,} \\
    0  & & \text{otherwise.}
  \end{array} \right. \nn \\
  \cal{F}\onen\la s\ra & \mapsto &
  \left\{\begin{array}{ccl}
    H_{k-1,k;N}\la s+1-k\ra & & \text{with $n=2k-N$ and $0< k \leq N$,} \\
    0  & & \text{otherwise.}
  \end{array} \right. \nn
\end{eqnarray}
It will be convenient in what follows to introduce a simplified notation in $\Gr$.
Corresponding to a fixed value of $N$ set $n=2k-N$ and write
\begin{align} \label{eq_shorthand}
  \1_n^N &:= \Gamma_N(\onen) \nn \\
  \rE \1_n^N &= \1_{n+2}^N\rE = \1_{n+2}^N\rE\1_n^N := \Gamma_{N}(\cal{E}\onen) \nn \\
  \rF \1_n^N &= \1_{n-2}^N\rF = \1_{n-2}^N\rF\1_n^N := \Gamma_N(\cal{F}\onen)
\end{align}
as a shorthand for the various bimodules.  Juxtaposition of these symbols represents the tensor product of the corresponding bimodules. For example,
\[
\rF \rE\rE\1_n^N = H_{k+1,k+2;N} \otimes_{H_{k+2;N}} H_{k+2,k+1;N} \otimes_{H_{k+1;N}} H_{k+1,k;N}.
\]
Associated to a signed sequence $\ep$  is the $(H_{k+|\ep|},H_k)$-bimodule
\[
\rE_{\ep}\1_n^N := \rE_{\epsilon_1} \rE_{\epsilon_2}\dots \rE_{\epsilon_m}\1_n^N
\]
where $\rE_{+}:= \rE$ and $\rE_{-}:= \rF$. The 2-functor $\Gamma_N$ maps a composite $\cal{E}_{\ep}\onen$ of 1-morphisms in $\UcatD$ to the tensor product $\rE_{\ep}\1_n^N$ in $\Gr$. Note that because tensor product of bimodules is only associative up to coherent isomorphism our notation is ambiguous unless we choose a parenthesization of the bimodules in question.  We employ the convention that all parenthesis are on the far left.  Hence, $\Gamma_N$ preserves composition of 1-morphisms only up to coherent 2-isomorphism.

It is sometimes convenient to use the isomorphism \eqref{eq_Hk-tensor} so that
\begin{align}
  \Gamma_N(\cal{E}^a\onen) &\cong H_{k+a,k+a-1, \dots, k;N} \la r_a\ra,
  \qquad r_a = \sum_{i=1}^{a}  i-N+k \\
  \Gamma_N(\cal{F}^a\onen) &\cong H_{k,k+1, \dots, k+a;N} \la r'_a \ra ,
  \qquad r'_a =\sum_{i=1}^{a} i-k.
\end{align}
We also define bimodules
\begin{align}
  \rE^{(a)}\1_n^N &:=H_{k+a,k;N}\la r_a -\frac{a(a-1)}{2}\ra, \\
  \rF^{(b)}\1_n^N &:=H_{k+a,k;N}\la r_a +\frac{a(a-1)}{2}\ra.
\end{align}
For $x \in \B$ write $\rE(x)\1_n^N$ as in \eqref{eq_basis} for the corresponding tensor product of bimodules.

For the images of 2-morphisms see \cite[Section 7]{Lau1}.  We note that the images of caps and cups give  $\rE\1_n^N$ and $\rF\1_{n+2}^N$ a biadjoint structure.  The image of dot on the $i$th strand corresponds to the bimodule endomorphism of $\rE^a\1_n^N$ given by multiplication by $\xi_i$ and a crossing of the $i$ and $i+1$st strand acts by fixing generators $x_{j,n}$ and $y_{j,n+2a}$ for all $j$ and mapping $f \in \Q[\xi_1,\dots,\xi_a]$ by divided difference operators
\[
 \partial_i(f) := \frac{f - s_i(f)}{x_i - x_{i+1}}.
\]

For any permutation $w$ in the symmetric group $S_a$ let $w=s_{i_1} s_{i_2} \dots s_{i_m}$ be a reduced expression of $w$ in terms of elementary transpositions. We write $\partial_w = \partial_{i_1} \dots \partial_{i_m}$.  It is clear from the definition of the divided difference operator that the image consists of polynomials that are symmetric in both $\xi_i$ and $\xi_{i+1}$.  If $w_0$ is the longest word in the symmetric group $S_a$, then  $\partial_{w_0}(f) \in \Q[\xi_1, \dots, \xi_a]^{S_a}$.

The forgetful map $G_{k,k+1,\dots,k+a}^N \to G_{k,k+a}^N$ gives an inclusion of rings $H_{k,k+a;N} \hookrightarrow H_{k,k+1,\dots,k+a;N}$.  Under this map the variables $x_{j,n}$ and $y_{\ell,n+2a}$ are mapped to themselves and $\varepsilon_j$ is mapped to the $j$th elementary symmetric function $e_j(\xi_1,\dots, \xi_a)$.  There is a map going the other direction that fixes $x_{j,n}$ and $y_{\ell,n+2a}$ and maps $f \in \Q[\xi_1, \dots, \xi_a]$ to $\partial_{w_0}(f) \in  \Q[\xi_1, \dots, \xi_a]^{S_a}$.

%
\subsubsection{Adjoints}
%

Recall that the left and right adjoints are defined in \cite{Lau1} as follows.
\begin{alignat}{2}
  \left( \cal{E}\onen \la s\ra \right)^L &=  \cal{F}\onenn{n+2} \la n+1-s \ra,
  & \qquad
  \left( \cal{E}\onen \la s\ra \right)^R &=   \cal{F}\onenn{n+2} \la-n-1-s \ra,\nn \\
    \left( \cal{F}\onen\la s\ra \right)^L &=   \cal{E}\onenn{n-2} \la -n+1-s \ra,
  & \qquad
  \left(\cal{F}\onen\la s\ra\right)^R &=  \cal{E}\onenn{n-2} \la n-1-s \ra.
\end{alignat}
The 2-functor $\Gamma_N$ transports these to adjoints in $\Gr$
\begin{alignat}{2}
  \left( \rE\1_n^N \la s\ra \right)^L &=  \rF\1_{n+2}^N \la n+1-s \ra,
  & \qquad
  \left( \rE\1_{n}^N \la s\ra \right)^R &=   \rF\1_{n+2}^N \la-n-1-s \ra,\nn \\
    \left( \rF\1_{n}^N\la s\ra \right)^L &=   \rE\1_{n-2}^N \la -n+1-s \ra,
  & \qquad
  \left(\rF\1_{n}^N\la s\ra\right)^R &=  \rE\1_{n-2}^N \la n-1-s \ra.
\end{alignat}
In what follows we make use of the fact that  $(u^L)^R=u$ and $(v^R)^L=v$ for all 1-morphism and that the adjoints give rise to isomorphisms
\begin{alignat}{3}
 \Hom_{\Gr}(ux,y) &\; \cong \;& \Hom_{\Gr} (x, u^R y)
 &\qquad \Hom_{\Gr}(x,uy)&\; \cong \;&  \Hom_{\Gr} (u^Lx,  y) \nn\\
 \Hom_{\Gr}(xv,y) &\; \cong \;& \Hom_{\Gr}(x,yv^L)  &\qquad \Hom_{\Gr}(x,yv) &\; \cong \;& \Hom_{\Gr}(xv^R,y).
\end{alignat}

%
\subsection{Karoubi completeness of $\Gr$}
%

In this section we show that the 2-category $\Gr$ is Karoubi complete.  Karoubi completeness implies that the 2-functor $\Gamma_N$ extends to a 2-functor $\Gamma_N \maps \UcatD \to \Gr$.  We prove this result directly, rather than passing to a Karoubi completion of $\Gr$, to maintain an explicit description of all 1-morphisms in the 2-category $\Gr$.  This will be needed to realize $\UcatD$ as an inverse limit of flag 2-categories.

It was shown in \cite{KLMS} (see also \cite{Lau1}) that the action of the nilHecke algebra provides an explicit decomposition of $\cal{E}^a\onen$ into a direct sum $\bigoplus_{[a]!}\cal{E}^{(a)} \onen$.  The projection onto the lowest degree summand $\cal{E}^{(a)} \onen \la -\frac{a(a-1)}{2}\ra $ is given by the image of the idempotent $e_a=x^{\delta}\partial_{w_0}:=x_1^{a-1}x_{a-1}^{a-2}\dots x_a^0\partial_{w_0}$, shown below for $a=4$.
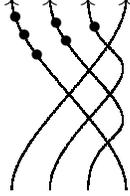
\[
\xy 0;/r.15pc/:
 (-12,-20)*{}; (12,20) **\crv{(-12,-8) & (12,8)}?(1)*\dir{>};
 (-4,-20)*{}; (4,20) **\crv{(-4,-13) & (12,2) & (12,8)&(4,13)}?(1)*\dir{>};?(.88)*\dir{}+(0.1,0)*{\bullet};
 (4,-20)*{}; (-4,20) **\crv{(4,-13) & (12,-8) & (12,-2)&(-4,13)}?(1)*\dir{>}?(.86)*\dir{}+(0.1,0)*{\bullet};
 ?(.92)*\dir{}+(0.1,0)*{\bullet};
 (12,-20)*{}; (-12,20) **\crv{(12,-8) & (-12,8)}?(1)*\dir{>}?(.70)*\dir{}+(0.1,0)*{\bullet};
 ?(.90)*\dir{}+(0.1,0)*{\bullet};?(.80)*\dir{}+(0.1,0)*{\bullet};
 \endxy
\]
The image of this idempotent under $\Gamma_N$ gives an idempotent 2-morphism $\Gamma_N(e_a) \maps \rE^a \1_n^N \to \rE^a \1_n^N$.

\begin{prop}
The idempotent 2-morphism $\Gamma_N(e_a) \maps \rE^a \1_n^N \to \rE^a \1_n^N$ splits in $\Gr$.
\end{prop}

\begin{proof}
The splitting is provided using the bimodule
$\rE^{(a)}\1_n^N$.  In particular, $\Gamma_N(e_a)$ is the composite
\begin{equation} \nn
  \xymatrix{
   \rE^a \1_n^N \ar[rr]^-{\partial_{w_0}} && \rE^{(a)}\1_n^N \la -\frac{a(a-1)}{2}\ra  \ar@{^{(}->}[rr]^-{\xi^{\delta}} & &
   \rE^a \1_n^N
  }
\end{equation}
where $\partial_{w_0}$ denotes the divided difference operator corresponding to the longest word $w_0 \in S_n$, and $\xi^{\delta}$ denotes the inclusion followed by the bimodule homomorphism of multiplication by $\xi^{\delta}= \xi_1^{a-1} \xi_2^{a-2} \dots \xi_a^0$.  This is a splitting because for $f \in \Q[\varepsilon_1,\dots,\varepsilon_a]$ the inclusion maps $f$ to a symmetric function in $\Q[\xi_1, \dots, \xi_a]^{S_a}$ so that $\partial_{w_0}(\xi^{\delta}f)=f$.
\end{proof}

\begin{cor}
There is an explicit decomposition of bimodules
\[
 \rE^a \1_n^N \cong \bigoplus_{[a]!} \rE^{(a)} \1_n^N
\]
in $\Gr$.
\end{cor}

\begin{proof}
Once the image of the idempotent $e_a$ has been identified in $\Gr$, the action of the nilHecke algebra provides explicit maps giving the decomposition.  See \cite[Section 2.5]{KLMS} for more details of this explicit decomposition.
\end{proof}

Having identified the images of divided powers $\cal{E}^{(a)}\onen$ and $\cal{F}^{(a)}\onen$ in $\Gr$ we can now make use of the results from Sections 5.1 and 5.2 of \cite{KLMS} using the 2-morphisms in $\UcatD$ to give explicit isomorphisms
\begin{align}
\rE^{(a)}\rE^{(b)}\1_n^N &\cong \bigoplus_{\qbins{a+b}{a}}\rE^{(a+b)}\1_n^N , \label{eq_tired1}\\
 \rE^{(a)}\rF^{(b)}\1_n^N&\cong
\bigoplus_{j=0}^{\min(a,b)}\bigoplus_{\qbins{a-b+n}{j}}\rF^{(b-j)}
\rE^{(a-j)}\1_n^N, \quad
\text{if $n\ge b-a$}, \label{eq_tired2}\\
 \rF^{(b)}\rE^{(a)}\1_n^N&\cong
\bigoplus_{j=0}^{\min(a,b)}\bigoplus_{\qbins{b-a-n}{j}}\rE^{(a-j)}\rF^{(b-j)}\1_n^N
\quad
\text{if $n\leq b-a$.} \label{eq_tired}
\end{align}

We now look at the space of bimodule endomorphisms of $\rE^{(a)}\1_n^N$ in order to show that these bimodules are indecomposable.  The main technique utilizes the biadjoint structure together with the relations above. Below all Homs are taken in $\Gr$.

\begin{prop}\label{prop.2.4} For $n=2k-N$, we have isomorphisms of graded abelian groups
\begin{equation}
 \END (\rE^{(a)}\1_{n}^N) \cong    H_{k;k+a}\otimes H_{k+a; N}.
 \nn \end{equation}
\end{prop}

\begin{proof}
The proof is by induction on $k$ and $a$.
For $k=0, a=1$, it is easy to verify the claim.
Indeed, applying biadjointness and ${\mathfrak {sl}}_2$ relations we have
\begin{align*}
  \End^m (\rE \1_{-N}^N) & \cong
  \Hom^m(\1_{-N}^N, (\rE \1_{-N}^N)^R \rE\1_{-N}^N) \\
  &\cong \Hom^m(1_N^N, \rF\rE \1_{-N}^N \la N-1\ra) \\
  &\cong
\Hom^{m+N-1}(1^N_{-N},\rE\rF \1_{-N}^N) \bigoplus_{j=0}^{N-1} \End^{m+2j}( \1_{-N}^N).
\end{align*}
Recall that the bimodule $\1_{-N}^N:=\Gamma_N(\onenn{-N})$ is just the cohomology ring $H_{0;N} \cong \Q$, so that the space of bimodule endomorphisms $\END(\1_{-N}^N)\cong \Q$.
Note also that $\rF\1_{-N}^N=0$.  Hence, the first summand  above vanishes.
The other summands are isomorphic to $H_{1;N}$, as desired.

The induction step works as follows. Let us assume $n\leq -a+1$, then
applying \eqref{eq_tired1}, \eqref{eq_tired} and biadjointness again, we get
\begin{align*}
  \Hom^m (\oplus_{[a]}{\rE^{(a)}}\1_{n}^N, \rE^{(a)}1_n) & \cong
  \Hom^m ({\rE^{(a-1)}} \1_{n}^N, (\rE \1_{n+2a-2}^N)^R \rE^{(a)}\1_{n}^N) \\
  &\cong \Hom^m (\rE^{(a-1)}\1_n^N, \rF  \rE^{(a)} \1_{n}^N \la -n-2a+1 \ra) \\
  &\cong
  \Hom^{m-n-2a+1}( \rE^{(a-1)}1^N_n,  \rE^{(a)}\rF\1_{n}^N)
\bigoplus_{[1-a-n]} \End^{m-n-2a+1}( \rE^{(a-1)} \1_{n}^N)
\\
& \cong
\Hom^{m-2n-2a+2}(\oplus_{[a]} \rE^{(a)}1^N_{n-2},  \rE^{(a)}\1_{n-2}^N)
\bigoplus_{[1-a-n]} \End^{m-n-2a+1}( \rE^{(a-1)} \1_{n}^N).
\end{align*}
 We are left to
check that after applying
 induction hypothesis
the graded dimension of the last formula is
$[a] {\rm rk}_q (H_{k;k+a} \otimes H_{k+a;N})$. Indeed,
\begin{align*}
 q^{2N-4k-2a+2} \;{\rm rk}_q \;  H_{k-1; k+a-1} \; {\rm rk}_q \;  H_{k+a-1;N} \;
+  \frac{ [N-2k-a+1]_{q^2}}{[a]_{q^2}} \;{\rm rk}_q \;  H_{k; k+a-1} \; {\rm rk}_q \;  H_{k+a-1;N}
\end{align*}
 is equal to  $ {\qbins{k+a}{k}}_{q^2} {\qbins{N}{k+a}}_{q^2}$.

The case of $n\geq -a+1$ is similar. Here one should consider
\[\Hom^m(\rE^{(a)}\1_n^N, \oplus_{[a]} \rE^{(a)}\1_n^N)=
\Hom^{m+n+1} (\rE^{(a)}\rF\1_{n+2}^N,  \rE^{(a-1)}\1_{n+2}^N)
\]
then apply \eqref{eq_tired2} and induction on $a$ and $N-k$.
\end{proof}

\begin{cor}
The bimodule $\rE^{(a)} \1_{n}^N$ with $n=2k-N$ for $0 \leq k \leq N$ is an indecomposable summand of $\rE^a\1_n^N$.
\end{cor}

\begin{proof}
By Proposition \ref{prop.2.4},
 its space of endomorphisms in degree zero
is one-dimensional.
\end{proof}

 Recall that $\rE(x):=\Gamma_N(\cal{E}(x))$ denotes the images in $\Gr$ of the indecomposable 1-morphisms in $\UcatD$ defined in \eqref{eq_basis}.

\begin{prop}
The nonzero bimodules $\rE(x)$ for $x$ in Lusztig canonical basis $\B$ are indecomposable.
\end{prop}

\begin{proof}
This follows from the graded rank computations above using the biadjoint structure and the relations \eqref{eq_tired1}--\eqref{eq_tired} in $\Gr$ following the arguments in \cite[Proposition 9.10]{Lau1}.
\end{proof}

\begin{thm} \label{thm_Karoubi}
The 2-category $\Gr$ is Karoubi complete.
\end{thm}

\begin{proof}
Knowing that the bimodules $\rE(x)$ are indecomposable in $\Gr$ and having the relations \eqref{eq_tired1}--\eqref{eq_tired} provides a decomposition of an arbitrary 1-morphism into indecomposables as in \cite[Proposition 9.10]{Lau1}.
\end{proof}

All 1-morphisms in $\Gr$ are finite direct sums of tensor products of bimodules of finite Krull dimension over rings with finite Krull dimension. Hence, the space of bimodule endomorphisms $\END_{\Gr}(x)$  of a 1-morphism $x$ in $\Gr$ is finite dimensional in each degree. This implies that the 2-category $\Gr$ has the Krull-Schmidt property so that any 1-morphism $x$ in $\Gr$ can be written as a unique direct sum of indecomposables.

%
\subsection{Properties of the 2-representations $\Gamma_N$}
%

In this section we show that the adjoint structure in $\Gr$ inherited from the adjoints in $\UcatD$ by the 2-functors $\Gamma_N$ strongly controls the size of Homs between 1-morphisms in the image of this 2-functor.

Recall that a functor is essentially surjective if every object in the target is isomorphic to an object in the image of the functor.  The functor is full/faithful if it is surjective/injective on Hom sets.  A functor that is full, faithful, and essentially surjective is an equivalence of categories, see for example~\cite[Section IV.4]{MacLane}.

\begin{defn}
A 2-functor $F \maps \cal{K} \to \cal{K}'$ is called locally full (respectively faithful, essentially surjective) if each functor on hom-categories $F \maps \cal{K}(x,y) \to \cal{K}'(Fx,Fy)$ is full (faithful/essentially surjective).
\end{defn}

\begin{prop} \label{prop_local_surjective}
The 2-functor $\Gamma_N$ is surjective on objects and locally essentially surjective.
\end{prop}

\begin{proof}
It is clear that $\Gamma_N$ is surjective on objects.  Using equation~\eqref{eq_Hk-tensor} it follows that every 1-morphism in $\Gr$ is isomorphic to a direct sum of bimodules corresponding to concatenated strings $\Gamma_N(\cal{E}_{\ep}\onen) := \rE_{\ep}\1_n^N$ proving local essential surjectivity.
\end{proof}

\begin{thm} \label{thm_full}
The 2-functors $\Gamma_N \maps \UcatD \to \Gr$ are locally full.
\end{thm}

\begin{proof}
We have shown that any 1-morphism in $\Gr$ can be decomposed into a direct sum of the bimodules $\rE(x)$ for $x \in \B$.  Hence it suffices to show that any 2-morphism in $\HOM_{\Gr}(\rE(x),\rE(x'))$ is in the image of $\Gamma_N$.  However, since the isomorphism in \eqref{eq_tired1}--\eqref{eq_tired} are realized through bimodule maps in the image of $\Gamma_N$, the space of maps $\HOM_{\Gr}(\rE(x),\rE(x'))$ are determined by the space of endomorphisms between 1-morphisms $\rE^{(a)}\1_n^N$ and $\rF^{(b)}\1_n^N$ since
\begin{alignat}{1} \label{5.1}
\HOM(\rE^{(a)}\rF^{(b)}\1_n^N, \rE^{(c)}\rF^{(d)}\1_n^N)=
\HOM(\rF^{(b)}\1_n^N, \rF^{(a)}\rE^{(c)}\rF^{(d)}\1_n^N)\la-a(n+2(c-d)+a)\ra\nn
\\
=\bigoplus^{\min{a,c}}_{j=0}
\bigoplus_{\qbins{a-c-n+2d}{j}}\HOM(\rF^{(b)}\1_n^N, \rE^{(c-j)}\rF^{(a-j)}
\rF^{(d)}\1_n^N)\la-a(n+2(c-d)+a)\ra \nn\\ =
\bigoplus^{\min{a,c}}_{j=0}
\bigoplus_{\qbins{b+d-n}{j}}\HOM(\rF^{(c-j)}\rF^{(b)}\1_n^N, \rF^{(a-j)}
\rF^{(d)}\1_n^N)\la(a+c-j)(b+d-j-n)\ra\, .\nn
\end{alignat}
Note that this space is zero unless $a-b=c-d$.
Finally,  observing that
$$ \rF^{(c)}\rF^{(d)}\1_n^N=\bigoplus_{\qbins{c+d}{c}}\rF^{(c+d)}\1_n^N$$
it follows that it is enough to prove fullness for
$\HOM (\rF^{(a)}\1_n^N, \rF^{(b)}\1_n^N)$ or
$\HOM (\rE^{(a)}\1_n^N, \rE^{(b)}\1_n^N)$, which are zero unless $a=b$.

It was shown by explicit computation in \cite[Proposition 8.1]{Lau1} that the $\Gamma_N$ maps $\END_{\UcatD}(\onen)$ surjectively onto $\END_{\Gr}(\1_n^{N})$ provided that 2 is invertible in the ground field $\Bbbk$.  Furthermore, from~\cite{KLMS}
\begin{equation} \label{eq_Ea}
\END_{\UcatD}(\cal{E}^{(a)}\onen) \cong \Lambda_a \otimes \Q[v_1,v_2,\dots ]
\end{equation}
where $\Lambda_a$ is the ring of symmetric functions in $a$ variables and $\Q[v_1,v_2, \dots]$ is the span of the dotted bubbles in $\END_{\UcatD}(\onen)$.   Therefore the image of $\END_{\UcatD}(\cal{E}^{(a)}\onen)$ under $\Gamma_N$ includes bimodule endomorphisms of $\rE^{(a)}\1_n^N$ given by multiplication by
products  of generators $\varepsilon_j$ and $x_{\ell,n}$  for $1\leq j \leq a$ and $1 \leq \ell \leq k$.  The space of such bimodule maps has graded rank equal to the graded rank of $H_{k,k+a;N}= \rE^{(a)}\1_n^N \la -r_a+\frac{a(a-1)}{2}\ra $.  The result then follows by a comparison of the graded ranks using Proposition~\ref{prop.2.4}
\[
 {\rm rk}_q \END (\rE^{(a)}\1_{n}^N) = {\rm rk}_q  \left(    H_{k;k+a}\otimes H_{k+a; N}\right)  =  {\rm rk}_q \left(H_{k,k+a;N} \right),
\]
where the last equality follows from the spectral sequence of a fibration $G_{k,k+a;N} \to G_{k+a;N}$, see for example \cite[Proposition 2.3]{GR}.
\end{proof}

Combining the above Theorem with certain faithfullness results from \cite{Lau1} we have the following.

\begin{cor}[Partial graded locally full and faithful] \label{cor_local_iso}
Given any two 1-morphisms
$x$ and $y$ of $\UcatD$ and  $d\in \N$, there exists
a positive integer $M$, such that the map
$$\Gamma_N: \Hom^d (x,y) \to \Hom_{\Gr \nn}^d (\Gamma_N(x), \Gamma_N(y))$$
is an isomorphism for any $N>M$.
\end{cor}

%
\section{The inverse limit Flag 2-category}
%

%
\subsubsection{Reminders on Inverse limits}
%

An inverse system $\varprojlim f$ in a category $\cal{C}$ is a family of homomorphisms $f_{ji} \maps C_j \to C_i$ for all $i \leq j$ such that $f_{ii}=\Id_{C_i}$ and $f_{ki} = f_{ji} \circ f_{kj}$ for all $i \leq j \leq k$.  It will be convenient to describe an inverse system by a functor $F \maps \cal{D} \to \cal{C}$, where  $\cal{D}$ is the category freely generated by the graph
\begin{equation} \label{eq_inverse_diagram}
 \xy
  (-30,0)*+{\bullet}="0"+(0,3)*{\scs i+1};
 (-15,0)*+{\bullet}="1"+(0,3)*{\scs i};
 (0,0)*+{\;\;\cdots\;\;}="2";
 (15,0)*+{\bullet}="3"+(0,3)*{\scs 1};
 (30,0)*+{\bullet}="4"+(0,3)*{ \scs 0};
 (-46,0)*+{\cdots};
  {\ar (-40,0);"0"};
   {\ar^{f_{i+1,i}} "0";"1"};
  {\ar "1";"2"};
  {\ar "2";"3"};
  {\ar^{f_{1,0}} "3";"4"};
 \endxy
\end{equation}
The functor $F$ picks an object of $\cal{C}$ for each node of the graph and a 1-morphism of $\cal{C}$ for each arrow.   If $\cal{C}$ admits limits the inverse limit of the system is the subobject of the direct product
\[
\varprojlim_{i \in I} C_{i}
= \{
 \prod_{i \in I} C_i \mid C_i = f_{ji}(C_j) \quad \text{for all $i \leq j$ in $I$}
\} .
\]
This construction works for rings, modules, and bimodules.

The inverse limit $\varprojlim C_i$ is equipped with maps $\pi_i \maps \varprojlim C_i \to C_i$ satisfying $\pi_i = f_{ji} \pi_{j}$.  The limit $\varprojlim C_i$ has the universal property that for any other object $A$ with maps $\alpha_i \maps A \to C_i$ satisfying $\alpha_i = f_{ji} \alpha_j$ there exists a unique map $A\to \varprojlim C_i$ making the  diagram
\[
 \xy
  (0,30)*+{A}="tt";
  (0,15)*+{\varprojlim C_i}="t";
  (-20,0)*+{C_j}="bl";
  (20,0)*+{C_i}="br";
  {\ar_{\pi_j} "t";"bl"};
  {\ar^{\pi_i} "t";"br"};
  {\ar_{f_{ji}} "bl";"br"};
  {\ar^{} "tt";"t"};
  {\ar@/^1pc/^{\alpha_i} "tt";"br"};
  {\ar@/_1pc/_{\alpha_j} "tt";"bl"};
 \endxy
\]
commute.

To generalize the inverse limit to a higher categorical context it is useful to reformulate this definition in a more abstract language. The functor $F\maps \cal{D} \to \cal{C}$ defining the inverse system is called a $\cal{D}$-diagram in $\cal{C}$. Let $\Delta_A \maps \cal{D} \to \cal{C}$ denote the constant functor mapping all objects of $\cal{D}$ to the object $A$ of $\cal{C}$ and all morphisms of $\cal{D}$ to $1_A$ in $\cal{C}$.  The data of $A$ together with projection maps $\alpha_i \maps A \to C_i$ can all be specified by a natural transformation $\alpha \maps \Delta_A \to F$.  Such a natural transformation is called a cone over the functor $F$.  The limit of $F$ is a universal cone, that is, a pair consisting of an object $L$ of $\cal{C}$ together with a natural transformation $\pi \maps \Delta_L \to F$ giving a bijection
\begin{equation} \label{eq_def_limit}
\Hom_{\cal{K}}(A,L) \to \cat{Cone}(A, F)
\end{equation}
induced by composition with $\pi$ for all objects $A$ of $\cal{C}$.  Here $\cat{Cone}(A,F)$ denotes the set of natural transformations $\Delta_A \to F$.  The universal property implies that whenever $A$ defines a cone over $F$ there is a map $A \to L$ in $\cal{C}$ compatible with the cone structure.

%
\subsection{Inverse limits in 2-categories}
%

In this section we describe the formalism for taking the inverse limit of the diagram in \eqref{eq_inverse_diagram} where each node represents a bicategory and each arrow represents a (pseudo) 2-functor between bicategories.

Limits involving categories typically require the machinery of 2-limits in a 2-category~\cite{BKPS,kel2} because categories organize themselves into the structure of a 2-category \cat{Cat} consisting of categories, functors, and natural transformations.  The correct universal property for such limits often utilizes this 2-categorical structure giving rise to limits  being described uniquely up to equivalence rather than up to isomorphism as is typically the case.  In some instances of 2-limits it is more natural to consider cones that only commute up to 2-isomorphism.

Here we are interested in a limit involving bicategories and 2-functors between them.  This suggests that 3-categorical limits \cite{Power} will have to be employed as bicategories most naturally organize themselves into a tricategory \cat{Bicat} consisting of bicategories, 2-functors, 2-natural transformations, and modifications, see \cite{GPS,Gray}.  Fortunately, however, the full force of limits in tricategories will not need to be employed in this article.

A surprising observation made by Steve Lack is that bicategories can be organized into a bicategory by replacing 2-natural transformations by a new notion he called \emph{icons}~\cite{Lack}.  This new notion of morphism between 2-functors gives rise to a bicategory structure $\cat{iBicat}$ with objects bicategories, morphisms 2-functors, and 2-morphisms icons.  In general the notion of an icon between a pair of 2-functors is too strict for most applications of bicategories.  However, because all of the 2-functors in this article behave especially well on objects, the notion of icons is quite natural.

The advantage of working with icons is that we are able to describe the universal property of the inverse limit as a 2-limit in the bicategory \cat{iBicat} rather than working with the tricategory \cat{Bicat}.   For further reference for bicategories and 2-limits see~\cite{Lack2,StreetF} and \cite[Chapter 7]{Bor}.  For a gentle introduction to 2-categorical limits in stricter framework see \cite[Section 2]{Hoff}.

%
\subsubsection{Lack's icons} \label{sec_Lack}
%
We briefly recall some definitions from higher category theory.  For more details see \cite[Chapter 7]{Bor} or \cite{Gray}.

A {\em pseudo-functor}, or 2-functor, $F \maps \cal{K} \to \cal{K'}$ between bicategories consists of
\begin{itemize}
  \item a function from objects of $\cal{K}$ to objects of $\cal{K'}$,
  \item functors $F\maps \Hom_\cal{K}(A,B) \to \Hom_{\cal{K}'}(FA,FB)$ for all objects $A$, $B$ of $\cal{K}$,
  \item natural isomorphisms to replace the equalities in the definition of a 2-functor
\[
 F_{f,g} \maps F(g)F(f) \to F(gf), \qquad F_A \maps 1_{FA} \to F(1_A),
\]
\end{itemize}
subject to three coherence axioms for associativity and left and right unit constraints.

A {\em pseudo-natural transformation}, or 2-natural transformation,  $\alpha \maps F\To G$ consists of
\begin{itemize}
  \item for each object $A$ of $\cal{K}$, a 1-morphism $\alpha_A \maps FA \to GA$, and
  \item for each 1-morphism $f\maps A \to B$, an isomorphism
\[
 \xy
  (-10,8)*+{FA}="tl";
(-10,-8)*+{GA}="tr";
(10,8)*+{FB}="bl";
  (10,-8)*+{GB}="br";
  {\ar_{\alpha_A} "tl";"tr"};
  {\ar^{Ff} "tl";"bl"};
  {\ar_{Gf} "tr";"br"};
  {\ar^{\alpha_B} "bl";"br"};
  {\ar@{=>}_{\sim} (3,3); (-3,-3)};
 \endxy
\]
\end{itemize}
subject to three coherence conditions making the above isomorphism respect composition and identities in $\cal{K}$ and behave appropriately with respect to 2-morphisms in $\cal{K}$.

If we require that the maps $\alpha_A \maps FA \to GA$ are always identity maps then we arrive at the definition of an {\em icon}.  In particular, the 2-functors $F$ and $G$ are required to agree on objects in order to admit an icon $\alpha \maps F \To G$.

A {\em modification} between pseudo-natural transformations $\gamma \maps \alpha \Rrightarrow \beta \maps F\To G$ consists of a 2-morphism $\gamma_A \maps \alpha_A \to \beta_A$ for each object $A$ of $\cal{K}$ subject to the obvious coherence condition for 1-morphisms in $\cal{K}$.

%
\subsubsection{2-limits in bicategories}
%

Let $F \maps \cal{D} \to \cal{K}$ be a 2-functor with $\cal{D}$ a small 2-category.  For each object $A$ of $\cal{K}$ define the constant 2-functor $\Delta_A \maps \cal{D} \to \cal{K}$ sending every object $D$ of $\cal{D}$ to the object $A$, all 1-morphisms in $\cal{D}$ to the identity $1_A$, and all 2-morphisms in $\cal{D}$ to the identity 2-morphism $\Id_{1_A}$. Let $\cat{2Cone}(A,F)$ denote the category of 2-natural transformations $\Delta_A \To F$  with morphisms given by modifications.

A 2-limit (sometimes called bilimit) of $F$ is a pair $(L,\pi)$ consisting of an object $L$ of $\cal{K}$ and a 2-natural transformation $\pi \maps \Delta_L \To F$ from the constant 2-functor on $L$ to the 2-functor $F$. This pair is required to be universal in the sense that the functor
\begin{equation} \label{eq_def_2limit}
\Hom_{\cal{K}}(A,L) \to \cat{2Cone}(A, F)
\end{equation}
induced by composition with $\pi$ is an equivalence of categories for every object $A$ in $\cal{K}$.
If $(L,\pi)$ and $(L',\pi')$ are both 2-limits for $F$, then $L$ and $L'$ must be equivalent in $\cal{K}$ (see~\cite[Proposition 7.4.5]{Bor}) meaning that there exists 1-morphisms $f\maps L \to L'$ and $f'\maps L' \to L$  such that $f'f \cong 1_L$ and $ff' \cong 1_{L'}$.

%
\subsubsection{Inverse 2-limits in bicategories}
%

For inverse 2-limits we take $\cal{D}$ to be the (strict) 2-category freely generated by the graph
\[
 \xy
  (-40,0)*+{D_{i+1}}="0";
 (-18,0)*+{D_i}="1";
 (0,0)*+{\;\;\cdots\;\;}="2";
 (18,0)*+{D_1}="3";
 (36,0)*+{D_0}="4";
 (-60,0)*+{\cdots};
  {\ar (-55,0);"0"};
   {\ar^{f_{i+1,i}} "0";"1"};
  {\ar "1";"2"};
  {\ar "2";"3"};
  {\ar^{f_{1,0}} "3";"4"};
 \endxy
\]
so that $\cal{D}$ only has identity 2-morphisms.   A 2-functor $F \maps \cal{D} \to \cal{K}$ defines a $\cal{D}$-diagram in $\cal{K}$, that is a collection of objects $F(D_i)$ in $\cal{K}$ for each object $D_i$ in $\cal{D}$ and 1-morphisms $F(f_{ij}) \maps F(D_i)\to F(D_j)$ for each 1-morphism of $\cal{D}$.  When the 2-functor $F$ is strict the composite $F(f_{jk}) \circ F(f_{ij})$ equals the 1-morphism $F(f_{jk} \circ f_{ij})$.  In general the composite $F(f_{jk}) \circ F(f_{ij})$ is only isomorphic to $F(f_{jk} \circ f_{ij})$.  In this article, we will be interested in strict diagrams so that $F$ is a strict 2-functor.

A 2-natural transformation $\alpha \maps \Delta_A \To F$, called a 2-cone,  assigns to each object $D_i$ of $\cal{D}$ a morphism
\[
 \alpha_i \maps \Delta_A(D_i) = A \to F(D_i)
\]
in $\cal{K}$, and to each morphism $f_{ji}\maps D_{j} \to D_{i}$ in $\cal{D}$ an isomorphism
\[
 \xy
  (-22,8)*+{\Delta_A(D_j)=A}="tl";
(-22,-8)*+{F(D_j)}="tr";
(22,8)*+{\Delta_A(D_i)=A}="bl";
  (22,-8)*+{F(D_i)}="br";
  {\ar_-{\alpha_j} "tl";"tr"};
  {\ar^{\Delta_A(f)=1_A} "tl";"bl"};
  {\ar_{F(f_{ji})} "tr";"br"};
  {\ar^-{\alpha_i} "bl";"br"};
  {\ar@{=>}_{\sim} (3,3); (-3,-3)};
 \endxy
\]
in $\cal{K}$.   Such a 2-natural transformation is called a 2-cone because it can be organized into a diagram
\[
 \xy
 (-5,20)*+{A}="t";
  (-68,0)*+{\cdots};
  (-46,0)*+{F(D_{i+1})}="0";
 (-18,0)*+{F(D_i)}="1";
 (0,0)*+{\;\;\cdots\;\;}="2";
 (18,0)*+{F(D_1)}="3";
 (36,0)*+{F(D_0)}="4";
  {\ar (-62,0);"0"};
   {\ar_{\scs  F(f_{i+1,i})} "0";"1"};
  {\ar "1";"2"};
  {\ar "2";"3"};
  {\ar_{\scs F(f_{1,0})} "3";"4"};
  {\ar_{\alpha_{i+1}} "t";"0"};
  {\ar_{\alpha_i} "t";"1"};
  {\ar_{\alpha_1} "t";"3"};
  {\ar^{\alpha_0} "t";"4"};
 \endxy
\]
where each triangle is filled with a 2-isomorphism in $\cal{K}$. A modification $\gamma \maps \alpha \Rrightarrow \beta \maps \Delta_A \To F$ gives a morphism of 2-cones at $A$.  That is, for each object $D_i$ a 2-morphism from $\alpha_i \to \beta_i$ satisfying some natural conditions.

The 2-limit of the 2-functor $F \maps \cal{D} \to \cal{K}$ is a 2-cone $(L,\pi)$ with universal property described by \eqref{eq_def_2limit}.  In particular, given any other 2-cone $(A,\alpha)$ there exists a 1-morphism $A \to L$, unique up to 2-isomorphism in $\cal{K}$, making the diagrams
\begin{equation} \label{eq_universal_2limit}
 \xy
  (0,30)*+{A}="tt";
  (0,15)*+{L}="t";
  (-20,0)*+{F(D_{j})}="bl";
  (20,0)*+{F(D_i)}="br";
  {\ar_{\pi_{j}} "t";"bl"};
  {\ar^{\pi_{i}} "t";"br"};
  {\ar_{F(f_{ji})} "bl";"br"};
  {\ar^{} "tt";"t"};
  {\ar@/^1pc/^{\alpha_i} "tt";"br"};
  {\ar@/_1pc/_{\alpha_j} "tt";"bl"};
 \endxy
\end{equation}
commute up to 2-isomorphism for all 1-morphisms $f_{ji}$ of $\cal{D}$.

%
\subsection{The inverse Flag 2-category}
%

From here on we take $\cal{K}=\cat{iBicat}$ and construct a $\cal{D}$-diagram in \cat{iBicat}.

%
\subsubsection{An inverse system in $\cat{iBiCat}$}
%

For $p \geq 0$ set $N'=N+2p$ and regard $\Grn{N'}$ as a sub 2-category of the 2-category \cat{Bim} of graded bimodules. Consider the quotient resulting from the following maps:
\begin{align} \label{eq_Psi-ring}
 \Psi_{N',N} \maps  H_{k+p;N'} & \longrightarrow H_{k;N} \nn \\
  x_{j,n} &\mapsto
  \left\{
  \begin{array}{ll}
    x_{j,n} & \text{if $j\leq k$,}\\
    0 &\text{otherwise,}
  \end{array}   \right.\nn\\
  y_{j,n} &\mapsto
  \left\{
  \begin{array}{ll}
    y_{j,n} & \text{if $j\leq N-k$,}\\
    0 & \text{otherwise,}
  \end{array}
  \right.
\end{align}
taken for all $n=2k-N=2(k+p)-N'$ and  $0 \leq k \leq N$.  These maps are ring homomorphisms.  By base extension these homomorphisms induce homogeneous maps
\begin{align} \label{eq_Psi_bimodule}
 \Psi_{N',N} \maps  \rE^{(a)} \1_n^{N'} & \longrightarrow \rE^{(a)} \1_n^{N} \nn \\
  x_{j,n} &\mapsto
  \left\{
  \begin{array}{ll}
    x_{j,n} & \text{if $j\leq k$,}\\
    0 &\text{otherwise,}
  \end{array}   \right.\nn\\
    \varepsilon_{j} &\mapsto \varepsilon_j, \nn
  \\
  y_{j,n+2a} &\mapsto
  \left\{
  \begin{array}{ll}
    y_{j,n+2a} & \text{if $j\leq N-k-a$,}\\
    0 & \text{otherwise,}
  \end{array}
  \right.
\end{align}
sending the $(H_{k+a+p;N'},H_{k+p;N'})$-bimodule $\rE^{(a)}\1_n^{N'}$ to the $(H_{k+a;N},H_{k;N})$-bimodule $\rE^{(a)} \1_n^N$.   Base extension applied to the bimodule $\rE_{\ep}\1_n^{N'}$ corresponding to a tensor product of bimodules in $\Grn{N'}$ gives a bimodule isomorphic to the bimodule $\rE_{\ep}\1_n^N$ in $\Gr$.

Using properties of base extension for module homomorphisms gives rise to a 2-functor
\[
\Psi_{N',N} \maps \Grn{N'} \to \Gr.
\]
On 2-morphisms $\Psi_{N',N}$ sends a 2-morphism $\Gamma_{N'}(D)$ of $\Grn{N'}$ to the 2-morphism $\Gamma_N(D)$ in $\Gr$.  For every parenthesization of the tensor product $\rE_{(\ep)}\1_n^{N'}$ there is a corresponding parenthesization of the tensor product $\rE_{(\ep)}\1_n^N$ in $\Gr$, so the 2-functors $\Psi_{N',N}$ are surjective at all levels and strictly preserve composition. These 2-functors also satisfy the relation
 \begin{equation} \label{eq_compatible}
\Psi_{N'',N'} \circ \Psi_{N',N} = \Psi_{N'',N},
\end{equation}
giving rise to a pair of inverse systems
\begin{align} \label{eq_inverse_sys_flag}
 \xy
  (-40,0)*+{\Grn{N+2}}="0";
 (-18,0)*+{\Grn{N}}="1";
 (0,0)*+{\;\;\cdots\;\;}="2";
 (18,0)*+{\Grn{x+2}}="3";
 (40,0)*+{\Grn{x}}="4";
 (-60,0)*+{\cdots};
  {\ar (-55,0);"0"};
   {\ar^{\Psi_{N+2,N}} "0";"1"};
  {\ar "1";"2"};
  {\ar "2";"3"};
  {\ar^{\Psi_{x+2,x}} "3";"4"};
 \endxy
\end{align}
of bicategories and 2-functors, where $x=0$ if the subscripts $N$ of $\Gr$ are even and $x=1$ if the subscripts are odd.    In particular, we have defined a pair of strict 2-functors $\Psi \maps \cal{D} \to \cat{iBicat}$.  In the next section we show that the 2-limits of these 2-functors exist.

%
\subsubsection{Constructing 2-cones over $\Psi$}
%

Recall that the weight of the ring $H_{k;N}$ is defined as $n=2k-N$.  For any $p \geq 0$ set $N'=N+2p$.  The quotient maps \eqref{eq_Psi-ring} together with the compatibility condition \eqref{eq_compatible} imply that the weight preserving ring homomorphisms
\[
 \Psi_{N',N} \maps  H_{k+p;N'} \longrightarrow H_{k;N}
\]
satisfy $\Psi_{N'',N'} \circ \Psi_{N',N} = \Psi_{N'',N}$ giving an inverse system of ring homomorphisms. We denote the resulting inverse limit as
\[
\hat{H}_{\mathbf{n}}:=\varprojlim_{N \in \N} H_{\frac{n+N}{2};N}.
\]
These rings form the objects of the inverse limit Flag 2-category.

If we write the generators of $H_{k;N}$ as $x_{j,n}(N)$ and $y_{\ell,n}(N)$ to emphasize the dependence on $N$, applying the same conventions for when these variables are zero as in Section~\ref{eq_cohomology}, then the ring $\hat{H}_{\mathbf{n}}$ has a spanning set given by the elements
\begin{align} \label{eq_inverse_elements}
 \hat{x}_{j,n} := \prod_{N \in \N} x_{j,n}(N) \qquad \text{for $j \geq 0$,} \qquad
 \hat{y}_{\ell,n} := \prod_{N \in \N} y_{j,n}(N) \qquad\text{for $\ell \geq 0$.}
\end{align}
These elements satisfy relations arising from equating homogenous terms in $t$ in the equation
\[
 \left(1+\hat{x}_{1,n}t + \hat{x}_{2,n}t^2+ \dots+ \hat{x}_{j,n}t^j + \dots\right)
  \left(1+\hat{y}_{1,n}t + \hat{y}_{2,n}t^2+ \dots +\hat{y}_{\ell,n}t^\ell  + \dots \right) =1.
\]

We now consider Hom categories in the inverse Flag 2-category. Set $n=2k-N$ and $n'=2k'-N$ and consider the graded categories
\[
\cal{C}_N^{n,n'} := \HOM_{\Gr}(H_{k;N},H_{k';N}).
\]
The 2-functor $\Psi_{N',N}$ induces functors $ \Psi_{N',N} \maps\cal{C}_{N'}^{n,n'} \to \cal{C}_N^{n,n'}$ satisfying $\Psi_{N'',N'} \circ \Psi_{N',N} = \Psi_{N'',N}$. The inverse limit of these categories
\[
\varprojlim \cal{C}^{n,n'}:= \varprojlim_{N\in \N} \cal{C}_N^{n,n'}
\]
has as objects sequences
\[
(x_N)_{N\in \N} := \left\{ x_N =\1_{n'}^Nx\1_n^N \in \cal{C}_N^{n,n'} \mid \xymatrix@1{x_N \ar[r]^-{\sim} &\Psi_{N',N}(x_{N'})} \right\}.
\]

Morphisms in the inverse limit category $\varprojlim \cal{C}^{n,n'}$ are sequences of 2-morphisms
\[
(D_N)_{N \in \N} := \left\{ D_N \maps x_N \to x'_N\right\}_{N \in \N}
\]
 such that
\[
 \xy
  (-15,8)*+{x_N}="tl";
(-15,-8)*+{x'_N}="tr";
(15,8)*+{\Psi_{N',N}(x_{N'})}="bl";
  (15,-8)*+{\Psi_{N',N}(x'_{N'})}="br";
  {\ar_-{D_N} "tl";"tr"};
  {\ar^-{\sim} "tl";"bl"};
  {\ar_-{\sim} "tr";"br"};
  {\ar^-{\Psi_{N',N}(D_{N'})} "bl";"br"};
 \endxy
\]
commutes in $\Gr$.  The limit admits projection functors
\begin{eqnarray}
  \pi_{N'} \maps \varprojlim \cal{C}^{n,n'} &\to& \cal{C}_{N'}^{n,n'} \\
  (x_N) &\mapsto& x_{N'} \nn \\
  (D_N) &\mapsto& D_{N'} \nn
\end{eqnarray}
for all $N' \in \N$ making the diagram
\[
 \xy
  (0,15)*+{\varprojlim \cal{C}^{n,n'}}="t";
  (-15,0)*+{\cal{C}_{N'}^{n,n'}}="bl";
  (15,0)*+{\cal{C}_N^{n,n'}}="br";
  {\ar_{\pi_{N'}} "t";"bl"};
  {\ar^{\pi_{N}} "t";"br"};
  {\ar_{\Psi_{N',N}} "bl";"br"};
  {\ar@{=>}^{\sim} (3,6); (-1,2)};
 \endxy
\]
commutes up to invertible natural transformation.

%
\subsubsection{Constructing the inverse 2-limit }
%
We can now organize the inverse limits discussed above into a 2-category that we will refer to as the inverse limit Flag 2-category.  Define a composition functor
\begin{eqnarray} \label{eq_Ccomp}
 \varprojlim \cal{C}^{n',n''} \times \varprojlim \cal{C}^{n,n'} &\longrightarrow& \varprojlim \cal{C}^{n,n''}  \\
(y_N) \times (x_N)
& \mapsto& \left(y_N x_N \right)_{N \in \N} \nn\\
(D'_N) \times (D_N) &\mapsto& (D'_N.D_N) \nn
 \end{eqnarray}
for each $n,n',n''$ in $\Gr$. Above $D'_N.D_N$ denotes the horizontal composition in $\Gr$ given by tensor product of bimodule homomorphisms.   For each triple  of objects in $\Gr$ there is a natural transformation making these composition functors associative up to natural isomorphism.
\begin{defn}
The 2-category $\varprojlim \Grn{}$ is the graded additive 2-category with
\begin{itemize}
  \item objects: Objects are the rings $\hat{H}_{\mathbf{n}}$ for $n \in \Z$.
  \item Hom categories: These are defined as
$
 \HOM_{\varprojlim \Grn{}}(\hat{H}_{\mathbf{n}},\hat{H}_{\mathbf{n'}}) := \varprojlim \cal{C}^{n,n'}
$.
\end{itemize}
The composition of 1-morphisms, and horizontal composition of 2-morphisms is described by the functor from \eqref{eq_Ccomp}.
\end{defn}

By Theorem \ref{thm_full} any 2-morphism $D'$ in $\Gr$ is $\Gamma_N(D)$, for some 2-morphism $D$ in $\UcatD$, composed with associativity and unit isomorphisms. Since the inverse system defining $\varprojlim \Grn{}$ stabilized in each degree $d$,  all of the 2-morphisms in $\varprojlim \Grn{}$ are bimodule homomorphism of the form
\[
 (D_N) := \left( \Gamma_{N'}(D) \right)_{N' \in \N} \quad \text{for each 2-morphism $D$ of $\UcatD$}.
\]

Utilizing the projections from the inverse limits discussed above it is possible to define projection 2-functors $\pi_N \maps \varprojlim \Grn{} \to \Gr$ sending $\hat{H}_{\mathbf{n}}$ to $H_{k;N}$, and $\varprojlim \cal{C}^{n,n'}$ to $\cal{C}_N^{n,n'}$. These 2-functors make the diagram
\[
 \xy
  (0,15)*+{\varprojlim \Grn{}}="t";
  (-15,0)*+{\Grn{N'}}="bl";
  (15,0)*+{\Gr}="br";
  {\ar_{\pi_{N'}} "t";"bl"};
  {\ar^{\pi_{N}} "t";"br"};
  {\ar_{\Psi_{N',N}} "bl";"br"};
  {\ar@{=>}^{\sim} (3,6); (-1,2)};
 \endxy
\]
commute up to invertible 2-morphisms in \cat{iBicat}. Thus we have defined a 2-cone for the diagram $\Psi$ constructed in \eqref{eq_inverse_sys_flag}.

\subsubsection{Universal property of inverse 2-limit}
%

Using the universal properties of the various inverse limits involved in the construction of $\varprojlim \Grn{}$ we show this 2-category has the universal property defining the inverse 2-limit. Given any other 2-cone $(\cal{K},\alpha)$
\[
 \xy
  (0,15)*+{\cal{K}}="t";
  (-15,0)*+{\Grn{N+2}}="bl";
  (15,0)*+{\Gr}="br";
  {\ar_{\alpha_{N+2}} "t";"bl"};
  {\ar^{\alpha_{N}} "t";"br"};
  {\ar_{\Psi_{N+2,N}} "bl";"br"};
  {\ar@{=>}^{\sim} (3,6); (-1,2)};
 \endxy
\]
in \cat{iBicat}, define a 2-functor $\hat{\alpha} \maps \cal{K} \to \varprojlim \Grn{}$ as follows:
\begin{itemize}
  \item On objects $\hat{\alpha}$ maps an object $A$ of $\cal{K}$ to the object $\prod_{N}\alpha_N(A)$.

  \item If $f \maps A \to B$ in $\cal{K}$, then $\hat{\alpha}(f) := \left(\alpha_N(f)\right)_{N \in \N}$.

  \item If $T \maps f \To g$ is a 2-morphism in $\cal{K}$, then $\hat{\alpha}(T) := \left( \alpha_N(T)\right)_{N \in \N}$.
\end{itemize}
This 2-functor satisfies the universal property described in \eqref{eq_universal_2limit}.

%
\subsection{$\UcatD$ as an inverse 2-limit of Flag 2-categories}
%

The 2-category $\UcatD$ is equipped with 2-functors $\Gamma_N \maps \UcatD \to \Gr$ for all $N$.  These 2-functors define a 2-cone
\[
 \xy
  (0,15)*+{\UcatD}="t";
  (-15,0)*+{\Grn{N+2}}="bl";
  (15,0)*+{\Gr}="br";
  {\ar_{\Gamma_{N+2}} "t";"bl"};
  {\ar^{\Gamma_{N}} "t";"br"};
  {\ar_{\Psi_{N+2,N}} "bl";"br"};
 \endxy
\]
where the triangles commute up to 2-isomorphism in \cat{iBicat}.  Hence, by the universal property of the inverse 2-limit there exists a 2-functor $\hat{\Gamma} \maps \UcatD \to \varprojlim \Grn{}$, unique up to 2-isomorphism in \cat{iBicat}, making the diagram
\[
 \xy
  (0,30)*+{\UcatD}="tt";
  (0,15)*+{\varprojlim \Grn{}}="t";
  (-20,0)*+{\Grn{N+2}}="bl";
  (20,0)*+{\Gr}="br";
  {\ar_{\pi_{N+2}} "t";"bl"};
  {\ar^{\pi_{N}} "t";"br"};
  {\ar_{\Psi_{N+2,N}} "bl";"br"};
  {\ar^{\hat{\Gamma}} "tt";"t"};
  {\ar@/^1pc/^{\Gamma_N} "tt";"br"};
  {\ar@/_1pc/_{\Gamma_{N+2}} "tt";"bl"};
 \endxy
\]
commute up to 2-isomorphism in \cat{iBicat}.

%

\begin{thm} \label{thm_equiv}
The 2-functor $\hat{\Gamma} \maps \UcatD \to \varprojlim \Grn{}$ is an equivalence of 2-categories in \cat{iBicat}.
\end{thm}

\begin{proof}
The 2-functor $\hat{\Gamma}$ is a bijection on objects.  This 2-functor is locally essentially surjective and locally full as a consequence of Proposition~\ref{prop_local_surjective}.  From the description of 2-morphisms in $\varprojlim \Grn{}$ Corollary~\ref{cor_local_iso} implies that $\hat{\Gamma}$ is locally full and faithful. The result follows.
\end{proof}

\begin{cor}
The 2-category $\UcatD$ is the unique 2-category up to equivalence in $\cat{iBicat}$ that is equipped with 2-representations to $\Gr$ for all $N$ that commute with the projection 2-functors $\Psi_{N',N}$.
\end{cor}

%
\section{Applications}
%

%
\subsection{The Image of the Casimir}
%
In this section we adopt shift convention from \cite{BKL} and set
\[
 \Hom(x \la s \ra, y \la t \ra ) =  \Hom^{s-t}(x,y).
\]

In \cite{BKL} the Casimir complex was defined as follows:
\begin{align}
 \cal{C}\onenn{n} &:=
  \xy
  (-50,0)*+{\left(\begin{array}{c}
      \scs  \cal{F}\cal{E} \onen \la2\ra\\ \scs \onen \la1+n\ra
          \end{array}\right) }="1";
  (0,0)*+{\und{\left(\begin{array}{c}
      \scs  \cal{F}\cal{E} \onen\\ \scs \cal{F}\cal{E} \onen
    \end{array}\right)} }="3";
  (60,0)*+{\left(\begin{array}{c}
      \scs \cal{F}\cal{E} \onen \la-2\ra\\ \scs \onen \la -n-1\ra
    \end{array}\right) }="5";
   {\ar^-{
  \left(
    \begin{array}{cc}
      \text{$\Udowndot\Uup$} & \Ucupr \\ & \\
      \text{$\Udown\Uupdot $} & \Ucupr \\
    \end{array}
  \right)
   } "1";"3"};
   {\ar^-{
  \left(
    \begin{array}{cc}
      -\;\Udown\Uupdot  &  \text{$\Udowndot\Uup $}  \\ & \\
      \Ucapl & -\;\Ucapl\\
    \end{array}
  \right)
   } "3";"5"};
 \endxy \nn \\
 \label{eq_casimir}
\end{align}
where we underlined the term  in zero homological degree.  It was also shown that this complex categorifies an integral idempotented version of the Casimir element for $\U$.

Let us denote by $B_{k;N}$ the following $(H_{k;N},H_{k;N})$-bimodule
  $$B_{k;N}:=H_{k,k+1;N}\otimes_{H_{k+1;N}} H_{k+1,k;N}\, .$$
Using the definition of the 2-functor $\Gamma_N$,
we easily compute
\begin{align}
\Gamma_N( \cal{C}\onenn{n}) &:=
  \xy
  (-50,0)*+{\left(\begin{array}{c}
      \scs  B_{k;N}  \la 3-N\ra\\ \scs H_{k;N} \la1+n\ra
          \end{array}\right) }="1";
  (0,0)*+{\und{\left(\begin{array}{c}
      \scs  B_{k;N} \la1-N\ra\\\scs B_{k;N}\la1-N\ra
    \end{array}\right)} }="3";
  (60,0)*+{\left(\begin{array}{c}
      \scs B_{k;N} \la-1-N\ra\\ \scs H_{k;N} \la -n-1\ra
    \end{array}\right) }="5";
   {\ar^-{
\Gamma_N  \left(
    \begin{array}{cc}
      \text{$\Udowndot\Uup$} & \Ucupr \\ & \\
      \text{$\Udown\Uupdot $} & \Ucupr \\
    \end{array}
  \right)
   } "1";"3"};
   {\ar^-{ \Gamma_N
  \left(
    \begin{array}{cc}
      -\;\Udown\Uupdot  &  \text{$\Udowndot\Uup $}  \\ & \\
      \Ucapl & -\;\Ucapl\\
    \end{array}
  \right)
   } "3";"5"};
 \endxy \nn \\
 \label{eq_casimir}
\end{align}

Let us define a simpler complex
$$K:= H_{k;N}\,\la1+N\ra\,[-1] \oplus H_{k;N}\,\la-1-N\ra\,[1] ~, $$
where $[a]$ denotes the shift in the  homological degree by $a$. This complex
has zero differentials, and two copies of shifted $\Gamma_N(\onen)$  in
homological degrees $-1$ and $1$. Note that the shift depends only
on $N$, but not
on $n$.

\begin{thm} \label{thm_cas} The complex $K$ is
quasi-isomorphic to $\Gamma_N(\cal C\onen)$.
\end{thm}


\begin{proof}
In the case when $n=N$, the claim is obvious, since $\cal{E}1_{N}=0$ and $B_{N,N}=0$.

For general $n=2k-N$, let us first define a chain map $f: K \to \Gamma_N(\cal C\onen)$
as follows.
\begin{align*} f_{-1}:&\; H_{k;N}\la 1+N\ra &\to &\left(\begin{array}{c}
       \scs B_{k;N}  \la 3-N\ra\\\scs  H_{k;N} \la1+n\ra
          \end{array}\right)&\; \quad f_0=0\quad\quad &f_1:& H_{k;N}\la -1-N\ra &\to& \left(\begin{array}{c}
      \scs  B_{k;N}  \la -1-N\ra\\ \scs H_{k;N} \la1+n\ra
          \end{array}\right)\\
& 1&\mapsto&
 \quad Y&&& 1&\mapsto& \left(\begin{array}{c}
     \scs 1\otimes 1\\ \scs 0
          \end{array}\right)
\end{align*}
where
\begin{equation} \label{eq_ker}
Y:={\left(\begin{array}{c}
(1 \otimes y_{N-k-1, n+2}) \Gamma_N\left(\text{$\Ucupr$}\right)(1)
 \\ \\  -y_{N-k,n}
    \end{array}\right) }\, .
\end{equation}
Here  the first entry of $Y$ is an element of
 $H_{k,k+1;N}\otimes H_{k+1,k;N} \la 3-N\ra$
 which is constructed as follows:
we take the image of $1\in H_{k;N}$ under the degree $2k$ bimodule map
$\Gamma_N \left(\text{$\Ucupr$}  \right)$ (given by  eq.(7.6) in \cite{Lau1})
and then multiply  the second tensor factor by $y_{N-k-1,n+2}$.
It is easy to check that the total degree of the map $f$ is zero.

Let us check that this map induces an isomorphism on homology.

For
\[
A:= \Gamma_N \left(
    \begin{array}{cc}
      \text{$\Udowndot\Uup$} & \Ucupr \\ & \\
      \text{$\Udown\Uupdot $} & \Ucupr \\
    \end{array}
  \right)
\]
we claim that the kernel of $A$ as an $(H_{k;N},H_{k;N})$-bimodule
is generated by $Y$ and hence, the homology of
$\Gamma_N(\cal C\onen)$ in degree $-1$ is isomorphic to
$H_{k;N}\la N+1\ra$.

Let us justify the claim.
Sliding $y_{N-k,n}$ from the $n$th to the $(n+2)$th region by using
 relation (6.19) in \cite{Lau1},
we get the following identities
\begin{align}\label{slid}
\Gamma_N \left(\Ucupr\right)( y_{N-k,n}) &=
  (1\otimes \xi \, y_{N-k-1, n+2}) \Gamma_N\left( \Ucupr \right)( 1) \\ \nn
&=
(\xi\otimes y_{N-k-1, n+2})\Gamma_N\left( \Ucupr \right)( 1)
\end{align}
since $y_{N-k,n+2}=0$. Thus, $AY=0$.
It remains to show that  any other term in the kernel is of the form
$x_{i,n} Y$ for  $0\leq i\leq k$.
Assume there exists
$Z=(z_1, z_2)^t$ such that $AZ=0$, then we have
\[\Gamma_N \left(\Udowndot\Uup\right) ( z_1)=
\Gamma_N\left(\Udown\Uupdot\right) ( z_1)= -\Gamma_N \left(\Ucupr\right)
(z_2).
\]
By  fullness
of $\Gamma_N$, $z_1$ has to be of the form $\Gamma_N(\Ucupr)$ up to
multiplication with $\xi$ and $x_{i,n}$'s. Using \eqref{slid}, we get the claim.

Analogously, using fullness,
it is easy to check that under $\Gamma_N$ the kernel
of the map given by the second matrix in \eqref{eq_casimir}
is contained in the image
of the first one. Hence, the homology in degree $0$ vanishes.

Finally, observe that the map
$$\Gamma_N (\Ucapl): B_{k;N}\la1-N\ra \to H_{k;N}\la-1-n\ra$$
defined by eq. (7.8) in \cite{Lau1} is surjective. We conclude
that also $f_{1}$ induces an isomorphism on homology.

\end{proof}

%

\end{document}